\documentclass[11pt]{amsart}
\usepackage{a4wide,enumerate,enumitem,amsmath,color}
\usepackage[hidelinks]{hyperref}
\usepackage{graphicx}
\usepackage{mathtools}
\usepackage{soul}
\usepackage{float}

\allowdisplaybreaks

\usepackage{tikz}
\usepackage{pgfplots}
\pgfplotsset{compat=1.18}

\let\pa\partial
\let\na\nabla
\let\eps\varepsilon
\newcommand{\N}{{\mathbb N}}
\newcommand{\R}{{\mathbb R}}

\newcommand{\D}{\mathcal{D}}
\newcommand{\Dd}{\mathrm{D}}
\newcommand{\Hi}{\mathcal{H}}
\newcommand{\T}{\mathcal{T}}
\newcommand{\E}{\mathcal{E}}

\newcommand{\Eint}{\E_{{\rm int}}}

\newcommand{\dist}{\mathrm{d}}
\newcommand{\di}{d}
\newcommand{\m}{\mathrm{m}}
\newcommand{\Mu}{\mathbb{M}^n_u}
\newcommand{\Mv}{\mathbb{M}_v}
\newcommand{\Ne}{\mathcal{N}}

\newtheorem{theorem}{Theorem}
\newtheorem{lemma}[theorem]{Lemma}
\newtheorem{proposition}[theorem]{Proposition}
\newtheorem{remark}[theorem]{Remark}
\newtheorem{corollary}[theorem]{Corollary}
\newtheorem{definition}{Definition}

\begin{document}

\title[Finite volume scheme for local sensing chemotaxis]{A finite volume scheme for the local sensing chemotaxis model}

\author[M. Herda]{Maxime Herda}
\address{Univ. Lille, CNRS, Inria, UMR 8524--Laboratoire Paul Painlev\'e, 59000 Lille}
\email{maxime.herda@inria.fr}

\author[A. Trescases]{Ariane Trescases}
\address{Institut de Math\'ematiques de Toulouse; UMR5219 - Université de Toulouse ; CNRS - UPS, F-31062 Toulouse Cedex 9, France}
\email{ariane.trescases@math.univ-toulouse.fr}

\author[A. Zurek]{Antoine Zurek}
\address{CNRS – Université de Montréal CRM – CNRS, Montréal, Canada and Universit\'e de Technologie de Compi\`egne, LMAC, 60200 Compi\`egne, France}
\email{antoine.zurek@utc.fr}

\date{\today}

\begin{abstract}
In this paper we design, analyze and simulate a finite volume scheme for a cross-diffusion system which models chemotaxis with local sensing. This system has the same {Lyapunov function (or entropy)} as the celebrated minimal Keller-Segel system, but unlike the latter, its solutions are known to exist globally in 2D. The long-time behavior of solutions is only partially understood which motivates numerical exploration with a reliable numerical method. We propose a linearly implicit, two-point flux finite volume approximation of the system. We show that the scheme preserves, at the discrete level, the main features of the continuous system, namely mass { conservation}, non-negativity of solution, entropy { dissipation}, and duality estimates. These properties allow us to prove the well-posedness, unconditional stability and convergence of the scheme. We also show rigorously that the scheme possesses an asymptotic preserving (AP) property in the quasi-stationary limit. We complement our analysis with thorough numerical experiments investigating convergence and AP properties of the scheme as well as its reliability with respect to stability properties of steady solutions.

\bigskip
		
		
\medskip
		

\end{abstract}
\maketitle

\tableofcontents

\section{Introduction}

\subsection{Presentation of the system}

In some bounded domain $\Omega \subset \R^\di$ ($\di \geq 1$) and for some time interval $(0,T)$ we consider the following system
\begin{align}\label{1.equ}
\pa_t u &= \Delta (\gamma(v) \, u), \quad \mbox{in }Q_T =\Omega\times (0,T),\\
\eps \pa_t v &= \delta \Delta v - \beta v+u,\quad \mbox{in }Q_T = \Omega \times (0,T),\label{1.eqv}
\end{align}
with $\eps$ a nonnegative constant, $\beta$ and $\delta$ some positive constants and where the function $\gamma$ is given by
\begin{align}\label{1.motility}
\gamma(s) = e^{-s}, \quad \forall s \geq 0.
\end{align}
The system is complemented by the following boundary and initial conditions
\begin{align}\label{1.BC}
&\na [ \gamma(v) u] \cdot \nu = \na v \cdot \nu = 0, \quad \mbox{on } \pa \Omega \times (0,T),\\
&u(0) = u^0, \quad \eps v(0) = \eps v^0, \quad \mbox{in }\Omega,\label{1.IC}
\end{align}
where $\nu$ is the exterior unit normal vector to $\pa \Omega$, and where $u^0$ and $v^0$ are nonnegative functions.

The system~\eqref{1.equ}--\eqref{1.eqv} is a subclass of the class of systems introduced by Keller and Segel in their seminal papers \cite{KeSe70,KeSe71a} in order to model chemotactic aggregation in cell colonies. This particular subclass corresponds to \emph{local sensing}, that is, the cell is not able to perceive a gradient in the concentration of chemoattractant by comparing the concentrations at two different locations (\emph{gradient sensing}) and instead relies only on the concentration at the point it is located to direct itself: for a modelling discussion on the subject see \cite{DKTY2019,KeSe71a,OtSt97}. The nonnegative quantities $u$ and $v$ correspond respectively to the cell density and the chemical concentration, the parameter $\delta$ is the diffusivity of the chemical, the parameter $\beta$ is its degradation rate, and the parameter $\eps$ indicates a typical relative order of magnitude of the time evolution of the source terms in~\eqref{1.eqv} over the time evolution of $v$: when $\eps$ vanishes, the system is called quasi-stationary (or sometimes parabolic-elliptic). Finally, the nonnegative function $\gamma$ represents the cell motility.
 
{The present paper concerns the numerical resolution and analysis of the system \eqref{1.equ}-\eqref{1.IC}.} In the last two decades a literature has emerged on the design and analysis of numerical schemes for { Keller-Segel-type systems, including the celebrated minimal Keller-Segel system}~\cite{AT21}.
The focus has been put on many aspects including preservation of the conservative and entropy diminishing  nature of the { system}, stability analysis, the possibility of handling non-structured meshes as well as high-order  discretizations. Finite difference schemes have been investigated in \cite{SaSu05, Sa09, LWZ18}. Finite volume schemes have been proposed starting with the seminal contribution \cite{Fil06}, followed by many authors \cite{CK08, ABS11, BCJ14, ZS17, AB17}. Finite element method \cite{Ma03, Sa07, BCW10, Sa12, SSKHT13, ZZZ16}, discontinuous Galerkin \cite{EK09, EI09, ZZLY16, LSY17, GLY19} and alternative numerical strategies \cite{BCR05, BCC08, Ep12, CEHK18, EpXi19} have been analysed as well. However, the previous numerical schemes do not concern chemotactic systems with local sensing as the one considered in the present manuscript.

Some methods have been proposed very recently in order to study the system~\eqref{1.equ}--\eqref{1.eqv} and its variant with logistic-type growth, that is, when~\eqref{1.equ} is replaced with
\begin{align}\label{1.equ_growth}
\pa_t u &= \Delta (\gamma(v) \, u) + r \, u(1-u), \quad \mbox{in }Q_T =\Omega\times (0,T),
\end{align}
with $r\ge 0$. In \cite{HeNeVa23}, a meshless method based on Generalized Finite Difference is developed to approximate the quasi-stationary case ($\eps=0$) for a quite broad range of motilities gamma, including in particular~\eqref{1.motility}. It is shown that this method is convergent when the growth parameter $r$ is large enough, a condition under which the solutions are known to relax towards the homogeneous steady state. We also mention the work \cite{BrPa24} where computer-assisted proofs are developed to study the steady states of the system~\eqref{1.eqv} and~\eqref{1.equ_growth} for one dimensional domains and with $r\ge 0$ (for quite arbitrary motilities $\gamma$): with this method the authors obtain at once the existence of non-homogeneous steady states and the validation of their numerical approximation, and provide bifurcation diagrams. 

Our goal here is to propose a numerical scheme for the evolutionary system~\eqref{1.equ}--\eqref{1.motility} providing a faithful approximation of the solutions in various regimes, including away from the homogeneous stabilization regime { and including both the full system ($\eps>0$) and the quasi-stationary asymptotic.}  To achieve this goal we aim for a scheme that preserves the main structural features of the equation (mass conservation, entropy dissipation, non-negativity of solutions), in order to ensure good stability properties, with a computational cost that is as low as possible.

\subsection{Description and novelty of our main results}

In this work, we introduce a two-point flux approximation (TPFA) finite volume scheme in space together with a one step implicit/explicit (IMEX) Euler scheme in time for the system~\eqref{1.equ}--\eqref{1.eqv}. This IMEX strategy, which was already used in \cite{LWZ18} for the minimal Keller-Segel system, allows us to decouple at the discrete level the equations~\eqref{1.equ} and~\eqref{1.eqv} and yields an efficient, w.r.t.~computational cost, numerical scheme. This efficiency is particularly important in order to numerically explore the infinite-time blow-up phenomenon satisfied by the solutions to the system~\eqref{1.equ}--\eqref{1.IC}.

Let us now introduce the main hypothesis and results of our study. In all this paper, we will work under the following hypothesis:
\begin{itemize}
\item[\textbf{(H1)}] Domain: $\Omega\subset\R^\di$ is a bounded and polyhedral domain ($\di=1$, $2$, or $3$).

\item[\textbf{(H2)}] Parameters: $\eps$ is a nonnegative constant and $\beta$ and $\delta$ are positive constants.

\item[\textbf{(H3)}] Cell motility: The function $\gamma$ is defined by~\eqref{1.motility}.

\item[\textbf{(H4)}] Initial data: The functions $u^0$ and $\eps v^0$ are nonnegative and belong respectively to $L^2(\Omega)$ and $H^1(\Omega)$.
\end{itemize}

{
\begin{remark} Let us precise that the scheme that we propose can be used for other cell motilities than  \eqref{1.motility} (see the numerical illustrations in Section~\ref{sec:DKTY_testcase} for an example), but the numerical analysis is limited to the exponential motility prescribed in \textbf{\emph{(H3)}}. Indeed, the core of our analysis relies on the entropy dissipation of the system \eqref{1.equ}--\eqref{1.motility} exposed (at the continuous level) in Section~\ref{sec.keyest}, which holds for the specific motility prescribed in~\eqref{1.motility}. Adapting the analysis - and possibly the numerical scheme - to other motilities will be the subject of a future work.
\end{remark}}

Under these hypothesis, we first prove in Theorem~\ref{thm.exi} that the scheme admits a unique positive solution at each time step. We also show that the solutions of the scheme satisfy a discrete counterpart of the entropy estimate (see~\eqref{1.EI} below) established in~\cite{BLT21,FuJi21a,JiWa20}. The proof of Theorem~\ref{thm.exi} strongly relies on the linearly implicit structure of our scheme. More precisely, this structure allows us to rewrite the scheme as linear systems of equations. Then, we prove that the associated matrices are invertible and admit some monotonicity properties ensuring the nonnegativity of the solutions. The fact that the scheme is entropy diminishing is proved by adapting at the discrete level the computations done in~\cite{BLT21,FuJi21a,JiWa20}.

Our second main result, exposed in Theorem~\ref{thm.conv}, is the (subsequential) convergence of the scheme when $\eps>0$. As often with finite volume methods~\cite{EGH00} for nonlinear PDEs, the convergence proof is done through a compactness argument which relies on some uniform estimates with respect to the size of the meshes in space and time. If our strategy is quite classical we need to establish an unusual (to our best knowledge) discrete duality estimate. In order to establish such estimate we introduce a convenient discrete $H^1$ dual norm and we adapt at the discrete level the computations of~\cite{BLT21}, see Section~\ref{sec.keyest} and Proposition~\ref{prop.dual} for more details. Then, thanks to this uniform estimate adequately combined with the entropy and mass estimates, we are able to prove the (subsequential) convergence of the scheme towards weak solutions to~\eqref{1.equ}--\eqref{1.IC} in the sense of Definition~\ref{def.weak} below.

Our third main result, Theorem~\ref{thm.AP}, concerns the asymptotic preserving (AP) property of our scheme in the quasi-stationary limit $\eps\to 0$. For this purpose, we consider a family of initial data $(u^0,v^0)$ indexed in $\eps\ge0$ and we introduce an additional hypothesis:
\begin{itemize}
\item[\textbf{(H5)}] Well-prepared initial data: The $\eps$-family of functions $(u^0,v^0)$ is uniformly bounded in $L^2(\Omega)\times H^1(\Omega)$, it is such that $u^0$ converges a.e. as $\eps\rightarrow0$ towards a limit still denoted $u^0$, and it satisfies the condition
\begin{equation}\label{cdtintau}
    \langle u^0 \rangle -\beta \langle v^0 \rangle= O(\sqrt{\eps}).
\end{equation}
\end{itemize}
The latter hypothesis is intimately connected to the initial layer in the limit $\eps\to0$. More precisely, one can easily show that \eqref{cdtintau} is equivalent to $\|\langle \partial_t v \rangle\|_{L^2(0,T)} = O(1)$, where $\langle\cdot\rangle$ denotes the space average. We refer to Section~\ref{sec:continuous} for details, and to Section~\ref{sec:num_AP} for numerical illustrations of the initial layer.

Then, under assumption \textbf{(H5)}, we show in Theorem~\ref{thm.AP} that when both the size of the meshes and the parameter $\eps$ in~\eqref{1.eqv} go to zero the solutions of the scheme converge up to a subsequence towards the weak solutions of the quasi-stationary system in the sense of Definition~\ref{def.weak}. A by-product of Theorem~\ref{thm.AP} is the (subsequential) convergence of the scheme when $\eps=0$ (without the additional assumption \textbf{(H5)} which becomes irrelevant in the quasi-stationary case, see Remark \ref{rk:byproduct}).

Finally, for space dimensions one and two, and under further regularity assumptions for the initial data, Corollary~\ref{cor.AP} states a stronger AP property in the sense that the whole sequence of numerical approximations converges (and not only a subsequence).

We have implemented the scheme in one and two space dimensions on unstructured meshes (see Section~\ref{sec:imp} for details) and performed thorough numerical experiments to illustrate and complement our theoretical results, through four testcases. In the first testcase we illustrate the second order convergence in space of the scheme{, for the error between the discrete solution and a projection of the continuous solution on the mesh}. {This super-convergence property of a TPFA finite volume scheme is expected, see for instance~\cite{DroNat18_super_conv}.} In the second testcase, we are interested in the asymptotic preserving property and we show the effect of the well-preparedness of the initial data in the sense of \textbf{(H5)} and in a stronger sense. In the third testcase, we investigate stability properties of stationary solutions. We show that linear stability properties of homogeneous equilibrium are preserved by the scheme. Finally we demonstrate the applicability of our scheme for motilities which do not enter our theoretical framework.

\subsection{Outline of the paper}

The paper is organized as follows. { Section~\ref{sec:continuous} focuses on the continuous system: after a brief literature review, we recall and develop some key notions and estimates in the continuous setting.} We introduce the numerical scheme and present the main results in Section~\ref{sec.main}. We prove that the scheme is well-defined and satisfy some first properties (mass conservation, entropy dissipation) in Section~\ref{sec.exiproof}, while in Section~\ref{sec.unifest} we establish some uniform estimates. Then, we show the convergence of the scheme in Section~\ref{sec.convproof}. Section~\ref{sec.AP} is dedicated to the AP property of the scheme. Finally, in Section~\ref{sec.numexp} we present the aforementioned numerical experiments.

\section{Insight from the continuous level}\label{sec:continuous}

\subsection{ Literature review}

{ On the theoretical side, the system~\eqref{1.equ}--\eqref{1.eqv} has been the subject of a wide recent literature treating questions such as global existence of solutions and their long-time behaviour.
}

The problem of existence of global solutions to the system~\eqref{1.equ}--\eqref{1.eqv} with~\eqref{1.BC}--\eqref{1.IC} has attracted much attention over the past decade, with the exciting outcome that solutions do exist globally for rather generic relevant motilities $\gamma$ in any dimension, see for example \cite{FuJi20,JiLa21} when $\eps=0$ and \cite{FuJi21a,JiLaZh22,FuSe22,DLTW23} when $\eps>0$. In particular, for $\gamma$ given by~\eqref{1.motility}, existence of global solutions is known in any dimension: see \cite{FuJi20,JiLa21} when $\eps=0$ and \cite{FuJi21a,JiLaZh22,FuSe22} when $\eps>0$ for a theory of classical solutions and see \cite{BLT21} for a theory of weak solutions.

In fact, the choice~\eqref{1.motility} is motivated by the studies \cite{FuJi20,BLT21,FuJi21a,JiWa20} where it is shown that in bi-dimensional domains for this specific choice of $\gamma$ the system~\eqref{1.equ}--\eqref{1.IC} exhibits a critical mass phenomenon somewhat reminiscent of the behavior of the minimal Keller-Segel system. For subcritical initial cell masses, solutions exist globally and remain bounded uniformly in time, and for supercritical initial cell masses, solutions exist globally and may blow up at infinite time. While the behavioral change happens at the exact same critical mass for the system~\eqref{1.equ}--\eqref{1.IC} as for the minimal Keller-Segel system, one noteworthy difference with the latter is that here the blow-up is delayed to infinity. We also note that in bi-dimensional domains the decaying exponential motility~\eqref{1.motility} seems critical, in the sense that for a positive-valued motility function $\gamma$ nonincreasing and decaying to zero at infinity with
\begin{equation*}
    \lim_{s\rightarrow\infty} e^{\alpha s} \gamma(s) = \infty, \quad \forall \alpha>0,
\end{equation*}
the global solutions of~\eqref{1.equ}--\eqref{1.eqv} are uniformly bounded and the blow-up is therefore ruled out \cite{FuJi21b,JiLaZh22}.

This similarity with the minimal Keller-Segel system in dimension two is not a mere coincidence and reflects a structural similarity (in any space dimension), that translates at the minimum into a common entropy functional for the two systems \cite{BLT21,FuJi21a,JiWa20}. This entropy functional will be used in the present analysis. Another similitude is that both systems share the same set of steady states. For this reason, the existence of non-homogeneous steady states as well as their qualitative and asymptotic behaviour have been extensively studied for about forty years, see for example \cite{ssSchaaf85} in one-dimensional domains,
\cite{ssWaWe02,ssSeSu00,ssPiWe06,ssPPV16,ssBCR21} in bi-dimensional domains, \cite{ssBil98,ssPiVa15,ssAgPi16,ssBCN17,ssBCF20} in higher-dimensional domains, and references therein.

Beyond well-posedness and boundedness results, and the construction of special blowing- up solutions in the supercritical regime, little is known about the qualitative behavior of the evolutionary system. We are aware of two qualitative studies, both in the case $\eps>0$. In dimension one, the existence and stability of monotone steady states with boundary spike has been explored in \cite{ssWaXu21} using bifurcation theory and numerical simulations. The formation of patterns has been studied numerically and theoretically in \cite{ChKi24} for a porous-medium type generalization of the system~\eqref{1.equ}--\eqref{1.eqv} with $\eps>0$. In particular for the exponential motility~\eqref{1.motility}, a linear stability analysis around the homogeneous steady state $\left(\mu,\frac\mu\beta\right)$ for a given $\mu>0$ (see also \cite[proof of Theorem 3.1]{DKTY2019}) suggests that pattern formation has to be expected under the condition
\begin{equation}\label{cdt.linearly.unstable}
     \mu>\beta + \delta \lambda_1(\Omega),
\end{equation}
where $\lambda_1(\Omega)$ is the principal eigenvalue of the Laplace operator $-\Delta$ on $\Omega$ with homogeneous Neumann boundary conditions. Furthermore, bi-dimensional numerical experiments, supplemented by heuristic arguments based on the study of the excitable density set, might suggest that under this condition we can expect formation of peaks.

{
As one can see from the short review above, the analysis for the system~\eqref{1.equ}--\eqref{1.eqv} for the cases $\eps>0$ or $\eps=0$ is mostly done separately. In fact, when a similar result exists for both cases, especially about existence or boundedness/blow-up, the result for the quasi-stationary case is in general more accessible than (and precedes in the literature) the one for the full parabolic case $\eps>0$. That is to say, the quasi-parabolic system appears to be a good approximation of the full system $\eps>0$, both seemingly reproducing its qualitative behaviour and more accessible analytically. This naturally motivates the study of the limit $\eps\longrightarrow0$. Many studies indeed tackle the analytical proof of a similar limit for related systems, mainly focusing on the minimal Keller-Segel system (see~\cite{Rac09,Lip14,WWX19,Miz19,KuOg20,KuOg20b,KuOg22,Lem13,OgSu23,NoSa24}; see as well~\cite{LWZ18,LXZ23} for numerical experiments and methods in the quasi-stationary asymptotics for the minimal Keller-Segel system). However, as far as we are aware of, this limit has not been studied for the continuous  system~\eqref{1.equ}--\eqref{1.eqv}.

In the present study, we recall some tools initially developed for showing the existence of solutions for the system~\eqref{1.equ}--\eqref{1.IC}, in particular entropy and duality estimates, and adapt them in a discrete setting. This is key to establish the well-posedness and the convergence of our scheme. We furthermore develop new estimates to prove analytically the quasi-stationary limit for the system~\eqref{1.equ}--\eqref{1.IC} that we also adapt to the discrete setting. This is used to prove that our scheme is AP in the quasi-stationary limit. Our final aim with this scheme is to test and observe the long-time behaviour partially described in the literature on the continuous system: the section \ref{sec.numexp}, dedicated to numerical experiments, explores in particular the pattern formation under (a discrete counterpart of) the condition \eqref{cdt.linearly.unstable} and questions the uniform boundedness of the solutions (as opposed to blow-up) in 2D.
}

\subsection{Notions of solution}

The notion of weak solutions for the system~\eqref{1.equ}--\eqref{1.IC} is defined as follows:

\begin{definition}[Weak solutions to~\eqref{1.equ}--\eqref{1.IC}]\label{def.weak} 
Assume \emph{\textbf{(H1)}--\textbf{(H4)}}. Let $T>0$ and let $(u,v)$ nonnegative with
\begin{align*}
u \in L^\infty(0,T;L^1(\Omega)), \quad v \in L^\infty(0,T;H^1(\Omega)), \quad \eps v \in H^1(0,T;L^2(\Omega)), \quad u \, \gamma(v) \in L^2(Q_T).
\end{align*}
We call $(u,v)$ a weak solution of~\eqref{1.equ}--\eqref{1.IC} on $(0,T)$ if for all $\varphi \in C^{\infty}_c(\overline{\Omega} \times [0,T))$ satisfying $\nabla \varphi \cdot \nu = 0$ on $\partial \Omega \times [0,T)$, we have
\begin{align}\label{1.weaku}
\int_{Q_T} u \, \pa_t \varphi \, \dist x\dist t + \int_\Omega u^0(x) \varphi(x,0) \, \dist x + \int_{Q_T} u \gamma(v) \, \Delta \varphi \, \dist x \dist t = 0,
\end{align}
and
\begin{align}\label{1.weakv}
\eps \int_{Q_T} v \pa_t \varphi \, \dist x\dist t +\eps \int_\Omega v^0(x) \, \varphi(x,0) \, \dist x = \delta \int_{Q_T} \nabla v \cdot \nabla \varphi \,\dist x \dist t + \int_{Q_T} (\beta \, v -u) \, \varphi \, \dist x\dist t.
\end{align}
In particular when $\eps=0$, $(u,v)$ is called a weak solution of the quasi-stationary system.
\end{definition}

\begin{remark}
Under the regularity assumptions of the above definition, the distributional formulation~\eqref{1.weaku}--\eqref{1.weakv} is equivalent to the formulation (1.4)--(1.6) given in~\cite[Definition 1.1]{BLT21}. Indeed, starting from~\eqref{1.weaku} and first testing with $\varphi$ homogeneous in space, we obtain that $\langle u(\cdot,t) \rangle=\langle u^0 \rangle$ for a.e. $t\in [0,T)$. Then for $\psi$ in $C^{\infty}_c(\overline{\Omega} \times [0,T))$ we use~\eqref{1.weaku} with $\varphi = (-\Delta)^{-1}\left(\psi-\langle\psi\rangle \right)$, where $(-\Delta)^{-1}$ has to be understood with homogeneous Neumann boundary conditions, and we get
\begin{eqnarray*}
&-&\int_{Q_T} \left(u \gamma(v)-\langle u \gamma(v)\rangle\right) \, \psi \,  \dist x\dist t\\
&=& \int_{Q_T} u \, \pa_t (-\Delta)^{-1}\left(\psi-\langle\psi\rangle \right) \,  \dist x\dist t + \int_\Omega u^0 \, (-\Delta)^{-1}\left(\psi(\cdot,0)-\langle\psi(\cdot,0)\rangle \right) \, \dist x\\
&=& \int_{Q_T} \left(u-\langle u \rangle \right) \, (-\Delta)^{-1}\pa_t \left(\psi-\langle\psi\rangle \right) \,  \dist x\dist t\\
&& \hspace{4cm}+ \int_\Omega \left(u^0-\langle u^0 \rangle \right) (-\Delta)^{-1}\left(\psi(\cdot,0)-\langle\psi(\cdot,0)\rangle \right) \, \dist x\\
&=& \int_{Q_T} (-\Delta)^{-1} \left(u-\langle u \rangle \right) \, \pa_t \psi \,  \dist x\dist t + \int_\Omega (-\Delta)^{-1} \left(u^0-\langle u^0 \rangle \right) \psi(x,0) \, \dist x.
\end{eqnarray*}
A density argument together with the regularity of $u \gamma(v)$ and $u^0$ allow to conclude that (1.4) and the first equality of (1.6) in~\cite[Definition 1.1]{BLT21} hold. The reverse direction of the equivalence follows from the same computations, and the treatment of the equation for the evolution of $v$ is standard.
Here, we chose to use the distributional formulation~\eqref{1.weaku}--\eqref{1.weakv} which appears naturally in the asymptotics of the numerical scheme.
\end{remark}

The proof of the AP property of the scheme relies on some uniform estimates with respect to the size of the meshes and the parameter $\eps$. While some of these estimates are a direct by-product of the estimates established for the convergence proof, an extra uniform estimate has to be obtained on the $L^2$ norm of $\pa_t v$. We first present this estimate in Section~\ref{sec.keyest} where for pedagogical purpose we establish it formally at the continuous level, highlighting the role of \textbf{(H5)}. We then adapt it at the discrete level in Proposition~\ref{prop.dtv}. Finally, the proof of the AP property relies on a compactness approach. However, the compactness properties of the solutions do not ensure that the whole sequence of solutions to the scheme converges toward the weak solutions of the quasi-stationary system. Therefore, we propose to consider a more restrictive class of solutions for the quasi-stationary system in which uniqueness can be proven. Following~\cite{BLT21}, we then consider a slightly stronger notion of solution, namely the weak-strong solution:

\begin{definition}[Weak-strong solution for the quasi-stationary system]\label{def.WSquasistationary} Assume \emph{\textbf{(H1)}--\textbf{(H4)}} and $\eps=0$. Let $T>0$, and let $(u,v)$ nonnegative. Then, $(u,v)$ is called a weak-strong solution of the quasi-stationary system on $(0,T)$ if $(u,v)$ is a weak solution of the quasi-stationary system in the sense of Definition~\ref{def.weak} which additionally satisfies
\begin{align*}
    \delta \Delta v - \beta v + u = 0, \quad \mbox{in }L^2(Q_T).
\end{align*}
\end{definition}

Of course, one needs to derive sufficient conditions ensuring the existence of such weak-strong solutions. The answer is (partially) given in~\cite{BLT21}. Indeed, we have the following result: 

\begin{proposition}[Uniqueness of weak-strong solutions] \label{prop.exiWS}
Assume \emph{\textbf{(H1)}--\textbf{(H4)}} with $\di=1$ or $2$ and $\eps=0$. Assume furthermore that $u^0 \in L^\infty(\Omega)$. Then there exists at most one weak-strong solution of the quasi-stationary system in the sense of Definition~\ref{def.WSquasistationary} such that
\begin{align}\label{1.reguniq}
    u \in L^\infty(0,T;L^q(\Omega)), \quad \mbox{for }q > 2.
\end{align}
\end{proposition}

\begin{remark}\label{Rk.weakstong}
In~\cite{BLT21}, the existence and uniqueness of the weak-strong solution is obtained for the system~\eqref{1.equ}--\eqref{1.IC} for $\eps>0$ and not for the quasi-stationary system. Besides, this property is only proved for a smooth bounded domain $\Omega$ in $\R^\di$ while in this work we consider only polygonal (or polyhedron) domains. We therefore explain here how to adapt the arguments of~\cite{BLT21} in our framework. First, under the regularity assumption~\eqref{1.reguniq} one can readily adapt the proof of uniqueness for the quasi-stationary system, see the proof of~\cite[Theorem 1.6]{BLT21}. Besides, in order to construct a weak-strong solution it is sufficient to improve the regularity of the weak solutions. The arguments given in the proof of~\cite[Proposition 1.7]{BLT21} to improve the regularity of the weak solution of~\eqref{1.equ}--\eqref{1.IC} for one dimensional domains can be easily adapted for the quasi-stationary system in order to obtain that $u$ and $v$ belong to $L^\infty(Q_T)$. However, in bi-dimensional domains, due to the regularity of $\pa \Omega$, we need to modify more deeply the arguments of~\cite[Section 4.3]{BLT21}. For this purpose, noting that $u$ belongs to $L^\infty(0,T;L\log L)$ (this is a consequence of the entropy inequality~\eqref{1.EI} with $\eps=0$), we apply the result in~\cite{Grog93} which ensures that $v\in L^\infty(Q_T)$. Then, using the duality estimate we deduce that $u\in L^2(Q_T)$ and we conclude thanks to~\cite[Theorem 4.3.1.4]{Grisvard} that $v \in L^2(0,T;H^2(\Omega))$. Then, by a slight adaption of the proof of~\cite[Lemma 4.2]{BLT21} we conclude that $u \in L^\infty(0,T;L^q(\Omega))$ for any $q \geq 2$.
\end{remark}

\subsection{Key estimates at the continuous level}\label{sec.keyest}

As previously exposed, the proofs of our main theorems rely on the adaptation at the discrete level of some uniform estimates that are for most of them rigorously established for the continuous system in~\cite{BLT21}, namely { mass conservation, entropy dissipation}, and duality estimates. 
The study of the limit $\eps\rightarrow 0$ with the proof of the AP property also requires additional estimates that are uniform in $\eps$. Therefore we establish some $\eps$-uniform estimates for the time derivative of $v$, that are apparently new, though inspired by the treatment of the quasi-stationary system in~\cite{BLT21}. To give a flavor of our proofs, we use this section to present these estimates formally at the continuous level. Let us then consider $(u,v)$ nonnegative such that $(u,v)$ is a strong solution of~\eqref{1.equ}--\eqref{1.IC} and assume that $(u,v)$ is as smooth as needed for the computations performed in this section. For the sake of simplicity we assume in this section $\eps>0$, though the case $\eps=0$ can be treated similarly with minor adaptations.\\

\noindent{\it Mass and entropy.} Let us first state the mass estimates. For this purpose, we recall that $\langle w\rangle$ denotes the mean value of $w$ on $\Omega$:
\begin{align*}
    \langle w \rangle := \frac{1}{\m(\Omega)} \int_\Omega w(x) \, \dist x, \quad \forall w \in L^1(\Omega),
\end{align*}
where $\m$ denotes the Lebesgue measure (in any space dimension). Then, the system preserves the nonnegativity of $u$ and by integrating~\eqref{1.equ} over the domain $Q_t$ for $t>0$ we get the preservation of its mean value along time, i.e.,
\begin{align}\label{1.massu}
    \langle u(t)\rangle = \frac{1}{\m(\Omega)} \, \|u(t)\|_{L^1(\Omega)} =  \langle u^0 \rangle, \quad \mbox{for }t >0.
\end{align}
Besides, $v$ is also a nonnegative function and by integrating~\eqref{1.eqv} over the domain $\Omega$ we have
\begin{align*}
    \frac{\dist}{\dist t} \left(\langle v(t) \rangle - \frac{\langle u^0 \rangle}{\beta} \right) = - \frac{\beta}{\eps} \, \left(\langle v(t) \rangle - \frac{\langle u^0 \rangle}{\beta} \right),
\end{align*}
so that its mean value verifies
\begin{align}\label{1.massv}
    \langle v(t) \rangle = \frac{\langle u^0 \rangle}{\beta} + e^{-\beta t/\eps} \, \left( \langle v^0 \rangle - \frac{\langle u^0 \rangle}{\beta} \right), \quad \mbox{for }t>0.
\end{align}

Now, as for the minimal Keller-Segel model, one can notice {\cite{BLT21}} that~\eqref{1.equ}--\eqref{1.motility} admits a 
{Lyapunov} functional:
\begin{align*}
    H(u,v) = \int_\Omega \left[u \log u - u + 1 + \frac{\beta}{2} |v|^2 -uv + \frac{\delta}{2} \left|\nabla v \right|^2 \right] \, \dist x.
\end{align*}
{Indeed, computing
\begin{eqnarray*}
    \frac{\dist}{\dist t} H(u,v) &=& \int_\Omega \log \left(u e^{-v}\right) \, \partial_t u \, \dist x + \int_\Omega \left[ \beta v - u - \delta \Delta v \right] \, \partial_t v \, \dist x \\
    &=& - \int_\Omega \frac{1}{u e^{-v}} \left| \nabla \left(u e^{-v}\right)\right|^2 \, \dist x - \eps \int_\Omega \left|\partial_t v\right|^2\, \dist x \le 0,
\end{eqnarray*}
we obtain that $H(u,v)$ is nonincreasing. Note that this computation crucially relies on the specific choice of $\gamma$ in~\eqref{1.motility}. Furthermore, we get} the following dissipation equality
\begin{equation}\label{1.EI}\begin{split}
H(u(t),v(t)) + 4 \int_0^t\int_\Omega \left|\nabla \sqrt{u e^{-v}} \right|^2 \,\dist x \dist s + \eps \int_0^t\int_\Omega \left|\partial_t v \right|^2 \,\dist x\dist s & \\
= H(u^0,v^0), & \quad \mbox{for }t>0.
\end{split}\end{equation}
However, let us emphasize that by itself this dissipation equality is of limited use, for example it does not directly give \emph{a priori} estimates of $v$ in $L^\infty((0,T);H^1(\Omega))$. This is due to the nonpositive quadratic term $-u v$ in the definition of $H$. In order to circumvent this issue the main idea is to make use of a duality estimate, in a strategy that we now describe.\\

\noindent{\it Combination of duality and entropy.} { As explained above, our objective here is to find a lower bound for the nonpositive quadratic term $-u v$ in the definition of $H$. Noting that in the definition of $H$ there also appears a (nonnegative) term equivalent to the $H^1$ norm of $v$, our strategy will be to find an estimate for $u-\langle u \rangle$ in $(H^1)'$, the dual space of $H^1$. Indeed, provided we have such estimates, we can control,
\begin{equation*}
    u v = \underbrace{(u-\langle u \rangle)}_{\in (H^1)'} \underbrace{\phantom{\langle} v \phantom{\rangle}}_{\in H^1} + \,  \langle u^0 \rangle \underbrace{\phantom{\langle} v \phantom{\rangle}}_{\in H^1} \in L^1.
\end{equation*}
Roughly speaking, the main argument to obtain this $(H^1)'$ estimate of $u-\langle u \rangle$ is a combination of the equation for $u$ and the relation
\begin{align}\label{eq:identity_Hdual}
\|f\|_{(H^1)'(\Omega)} = \|\nabla (-\Delta)^{-1} f\|_{L^2(\Omega)},
\end{align}
valid for any $f$ with $\int_\Omega f(x) \, \dist x = 0$. To be more precise, let us first introduce some notations here:
}
 let $(H^1)'(\Omega)$ denote the dual space of $H^1(\Omega)$, let $H^1_{\Diamond}(\Omega)$ denote the set of all functions in $H^1(\Omega)$ with mean value zero, let $(H^1_{\Diamond})'(\Omega)$ denote its dual space, and finally let us introduce the operator $\mathcal{K}=(-\Delta)^{-1} : (H^1_{\Diamond})'(\Omega) \to H^1_{\Diamond}(\Omega)$ which inverses the Laplace operator with homogeneous Neumann boundary conditions.

Following~\cite{BLT21} by applying $\mathcal{K}$ to~\eqref{1.equ} and testing the result with $u-\langle u^0\rangle$ (which is of zero mean thanks to the mass conservation~\eqref{1.massu}) we obtain,
\begin{equation}\label{1.dual}\begin{split}
\frac12 \, \sup_{[0,T]} \, \|u-\langle u^0\rangle \|^2_{(H^1)'(\Omega)} &+ \int_0^T \int_\Omega \gamma(v) \, u^2 \,\dist x \dist t \\
&\leq \frac12 \|u^0-\langle u^0\rangle\|^2_{(H^1)'(\Omega)} + \m(\Omega) \langle u^0\rangle^2 \, T.
\end{split}\end{equation}

Then, denoting by $C_T>0$ the constant such that
\begin{align*}
    \|u(t)-\langle u^0\rangle \|_{(H^1)'(\Omega)} \leq C_T, \quad \mbox{for }t\in(0,T),
\end{align*}
we have, for $t\in(0,T)$,
\begin{align*}
    -\int_\Omega u \, v \, \dist x &= -\int_\Omega (u-\langle u^0 \rangle) \, v \,\dist x - \int_\Omega \langle u^0 \rangle \, v \, \dist x\\
    &= \int_\Omega v \, \Delta \mathcal{K} (u-\langle u^0 \rangle) \, \dist x - \int_\Omega \langle u^0 \rangle \, v \,\dist x\\
    &=- \int_\Omega \nabla \mathcal{K} (u-\langle u^0 \rangle) \cdot \nabla v \, \dist x - \int_\Omega \langle u^0 \rangle \, v \, \dist x\\
    &\geq -\frac{C_T^2}{\delta} - \frac{\delta}{4} \int_\Omega |\nabla v|^2 \, \dist x - \frac{m^2}{\beta} - \frac{\beta}{4} \int_\Omega |v|^2 \, \dist x.
\end{align*}
This yields, for $t\in(0,T)$,
\begin{align}\label{1.lowerbound_ent}
    H(u,v) \geq -\frac{C_T^2}{\delta^2} - \frac{\langle u^0\rangle^2}{\beta} + \int_\Omega \left[ u\log u - u+1 + \frac{\beta}{4} \, |v|^2 + \frac{\delta}{4} |\nabla v|^2 \right] \, \dist x.
\end{align}
Thus combining~\eqref{1.EI} and~\eqref{1.lowerbound_ent} one can deduce an upper bound on the $H^1(\Omega)$ norm of $v$ and on the logarithmic entropy of $u$, which is uniform over any finite time interval $(0,T)$. These estimates (in their discretized versions, and established rigorously), will be at the cornerstone of our convergence proof, see Theorem~\ref{thm.conv}. Besides, contrary to the minimal Keller-Segel model, these estimates also imply that if a blow-up phenomenon occurs it has to be postponed to the infinite time limit.\\

\noindent{\it Estimate in the quasi-stationary relaxation.} As already exposed, we intend in our final main result, see Theorem~\ref{thm.AP}, to consider the limit $\eps \to 0$. For this purpose we need to establish some estimates that are uniform with respect to $\eps$. However, we notice that the entropy dissipation equality~\eqref{1.EI} (combined with~\eqref{1.lowerbound_ent}) yields a $L^2(Q_T)$ estimate on $\sqrt{\eps} \partial_t v$ only. In order to derive an estimate on $\partial_t v$ uniform in $\eps$, we first differentiate~\eqref{1.eqv} w.r.t. time and get
\begin{eqnarray*}
\eps \partial_{tt}v + \beta \partial_t v = \partial_t u + \delta \Delta \partial_t v.
\end{eqnarray*}
Note that the right-hand-side having zero-mean (thanks to the mass conservation~\eqref{1.massu}) we can apply the operator $\mathcal{K}=(-\Delta)^{-1}$ to the equality above so that, after testing the result with $\left(\partial_t v-\langle \partial_t v\rangle\right)$ we obtain
\begin{align*}
\int_\Omega \left(\partial_t v-\langle \partial_t v\rangle\right) \mathcal{K} \big[\eps \partial_{tt}v &+ \beta \partial_t v\big]\,\dist x\\
&= \int_\Omega \left(\partial_t v-\langle \partial_t v\rangle\right) \mathcal{K} \left[\partial_t u\right]\,\dist x + \delta \int_\Omega \left(\partial_t v-\langle \partial_t v\rangle\right) \mathcal{K} \left[ \Delta \partial_t v\right]\,\dist x\\
&= -\int_\Omega \left(\partial_t v-\langle \partial_t v\rangle\right) \left(u\gamma(v) - \langle u \gamma(v)\rangle\right) \,\dist x - \delta \int_\Omega \left(\partial_t v-\langle \partial_t v\rangle\right)^2\,\dist x\\
&\le -\frac{\delta}{2} \int_\Omega \left(\partial_t v-\langle \partial_t v\rangle\right)^2\,\dist x +\frac{1}{2 \delta} \int_\Omega \left(u\gamma(v) - \langle u \gamma(v)\rangle\right)^2\,\dist x,
\end{align*}
where we have used the equation~\eqref{1.equ} for the second equality, and Young's inequality for the last inequality. We furthermore compute the left-hand-side as
\begin{eqnarray*}
&&\int_\Omega \left(\partial_t v-\langle \partial_t v\rangle\right) \mathcal{K} \left[\eps \partial_{tt}v + \beta \partial_t v\right]\,\dist x\\
&=& \eps \int_\Omega \left(\partial_t v-\langle \partial_t v\rangle\right) \mathcal{K} \left[ \partial_{tt}v - \langle \partial_{tt}v \rangle \right]\,\dist x + \beta \int_\Omega \left(\partial_t v-\langle \partial_t v\rangle\right) \mathcal{K} \left[ \partial_t v - \langle \partial_t v \rangle\right]\,\dist x\\
&=&  \eps \int_\Omega \nabla \mathcal{K} \left[\partial_t v-\langle \partial_t v\rangle\right] \cdot \nabla \mathcal{K} \left[ \partial_{tt}v - \langle \partial_{tt}v \rangle \right]\,\dist x + \beta \int_\Omega \left|\nabla \mathcal{K} \left[ \partial_t v - \langle \partial_t v \rangle\right] \right|^2\,\dist x\\
&=&  \frac{\eps}{2} \frac{\textrm{d}}{\textrm{d}t}\int_\Omega \left|\nabla \mathcal{K} \left[\partial_t v-\langle \partial_t v\rangle\right] \right|^2\,\dist x  + \beta \int_\Omega \left|\nabla \mathcal{K} \left[ \partial_t v - \langle \partial_t v \rangle\right] \right|^2\,\dist x,
\end{eqnarray*}
where we have used the definition of $\mathcal{K}$ (including the associated boundary conditions).
Summing up, we have that
\begin{multline}\label{estim.dtv}
\frac{\delta}{2} \int_\Omega \left(\partial_t v-\langle \partial_t v\rangle\right)^2\,\dist x
+ \frac{\eps}{2} \frac{\textrm{d}}{\textrm{d}t}\int_\Omega \left|\nabla \mathcal{K} \left[\partial_t v-\langle \partial_t v\rangle\right] \right|^2\,\dist x + \beta \int_\Omega \left|\nabla \mathcal{K} \left[ \partial_t v - \langle \partial_t v \rangle\right] \right|^2\,\dist x\\
\le
\frac{1}{2 \delta} \int_\Omega \left(u\gamma(v) - \langle u \gamma(v)\rangle\right)^2\,\dist x.
\end{multline}
Using that $u \gamma(v)$ is bounded in $L^2(Q_T)$ thanks to the duality estimate~\eqref{1.dual} and the fact that $\gamma$ in~\eqref{1.motility} is uniformly bounded, we finally get that, for any $T>0$,
\begin{eqnarray*}
\int_0^T \int_\Omega \left(\partial_t v-\langle \partial_t v\rangle\right)^2 \, \dist x \dist t
\le
C_T'.
\end{eqnarray*}
In order to estimate the $L^2$ norm of $\langle \pa_t v \rangle$, we differentiate~\eqref{1.massv} with respect to time to get,
\begin{align}\label{eq:evoldtV}
\langle \partial_t v\rangle = \frac{\langle u^0\rangle-\beta\langle v^0\rangle }{\eps} e^{-\frac{\beta}{\eps}t},
\end{align}
so that
\begin{eqnarray*}
\int_0^T \int_\Omega \langle \partial_t v\rangle^2 \, \dist x \dist t
= \frac{\m(\Omega)}{2\beta} \, \left(\frac{\langle u^0 \rangle -\beta \,\langle v^0\rangle}{\sqrt{\eps}} \right)^2 \left(1-e^{-\frac{2\beta}{\eps}T}\right).
\end{eqnarray*}
Therefore we conclude that $\partial_t v$ is bounded in $L^2(Q_T)$ uniformly in $\eps$ exactly under the assumption \textbf{(H5)}.

\section{Numerical scheme and main results}\label{sec.main}

\subsection{Meshes and discrete functional framework}\label{sec.mesh}

From now on $\di=1$, $2$ or $3$. For presentation purposes we present the discretization of the domain $Q_T=\Omega\times(0,T)$ and our results for bi-dimensional domains only. Unless stated otherwise, the generalization to other space dimensions is straightforward.

Let then $\Omega\subset\R^2$ be a bounded, 
polygonal domain. An admissible mesh of $\Omega$ is given by (i) a family $\T$
of open polygonal control volumes (or cells), (ii) a family $\E$ of edges, and
(iii) a family ${\mathcal P}$ of points $(x_K)_{K\in\T}$ associated to the
control volumes and satisfying Definition 9.1 in \cite{EGH00}. This definition
implies that the straight line $\overline{x_Kx_L}$ between two centers of
neighboring cells is orthogonal to the edge $\sigma=K|L$ between two cells.
For instance, Vorono\"{\i} meshes satisfy this condition \cite[Example 9.2]{EGH00}.
The size of the mesh is denoted by $\Delta x = \max_{K\in\T}\operatorname{diam}(K)$.
The family of edges $\E$ is assumed to consist of interior edges $\E_{\rm int}$
satisfying $\sigma\in\Omega$ and boundary edges $\sigma\in\E_{\rm ext}$ satisfying
$\sigma\subset\pa\Omega$. For a given $K\in\T$, $\E_K$ is the set of edges of $K$,
and it splits into $\E_K=\E_{{\rm int},K}\cup\E_{{\rm ext},K}$ and $\Ne(K)$ the set of its neighbouring control volumes (i.e. the cell $L$ such that $\overline{L}\cap\overline{K}$ is an edge of the discretization). For any $\sigma\in\E$, there exists at least one cell $K\in\T$ such that 
$\sigma\in\E_K$. Moreover, for $\sigma\in\E$, we introduce the distance
\begin{align*}
  \dist_\sigma = \begin{cases}
	\mathrm{dist}(x_K,x_L) &\quad\mbox{if }\sigma=K|L\in\E_{{\rm int},K}, \\
		\mathrm{dist}(x_K,\sigma) &\quad\mbox{if }\sigma\in\E_{{\rm ext},K},
	\end{cases}
\end{align*}
where $	\mathrm{dist}$ is the Euclidean distance in $\R^2$, and the transmissibility coefficient
\begin{equation}\label{2.trans}
  \tau_\sigma = \frac{\m(\sigma)}{\dist_\sigma}.
\end{equation}
The mesh is assumed to satisfy the following regularity assumption: There exists
$\zeta>0$ such that for all $K\in\T$ and $\sigma\in\E_K$,
\begin{equation}\label{2.dd}
  	\mathrm{dist}(x_K,\sigma)\ge \zeta \dist_\sigma.
\end{equation}

We also introduce suitable function spaces for the numerical scheme. The space of
piecewise constant functions is defined by 
\begin{align}\label{def:H_T}
  \Hi_{\T} = \bigg\lbrace w: \Omega\to\R:\exists (w_K)_{K\in\T}\subset\R,\
	w(x)=\sum_{K\in\T}w_K\mathbf{1}_K(x)\bigg\rbrace,
\end{align}
where $\mathbf{1}_K$ is the characteristic function on $K$. In the sequel, we will identify each $w \in \Hi_\T$ with its associated sequence $(w_K)_{K\in\T}$ and for any element $w \in \Hi_\T$, we will write (as at the continuous level)
\begin{align*}
\langle w \rangle = \frac{1}{\m(\Omega)} \sum_{K\in \T} \m(K) w_K.
\end{align*}
Besides, in order to define a norm on the space $\Hi_\T$, we first introduce the notation
\begin{align*}
  w_{K,\sigma} = \begin{cases}
	w_L &\quad\mbox{if }\sigma=K|L\in\Eint, \\
	w_K &\quad\mbox{if }\sigma\in\E_{{\rm ext},K},
	\end{cases} 
\end{align*}
for $K\in\T$, $\sigma\in\E_K$ and the discrete operators
\begin{align*}
  \textrm{D}_{K,\sigma} w := w_{K,\sigma}-w_K, \quad 
	\textrm{D}_\sigma w := |\mathrm{D}_{K,\sigma} w|.
\end{align*}
Let $q\in[1,\infty)$ and $w\in\Hi_{\T}$. The discrete $W^{1,q}$ seminorm and discrete $W^{1,q}$ norm on $\Hi_{\T}$ are given by 
\begin{align*}
  |w|_{1,q,\T}^q = \sum_{\sigma\in\E}\m(\sigma)\,\dist_{\sigma} 
	\bigg|\frac{\textrm{D}_{\sigma} w}{\dist_{\sigma}}\bigg|^q, \quad 
	\|w\|_{1,q,\T}^q = |w|^q_{1,q,\T} + \|w\|^q_{L^q(\Omega)},
\end{align*}
respectively, and $\|w\|_{L^q(\Omega)}$ denotes the $L^q$ norm 
i.e. $\|w\|_{L^q(\Omega)} = (\sum_{K\in\T}\m(K)|w_K|^q)^{1/q}$.

We finally define the space-time discrete framework. Let $T>0$, let $N_T\in\N$ be the number of time steps. We introduce the
step size $\Delta t=T/N_T$ as well as the time steps $t_n=n\Delta t$ 
for $n=0,\ldots,N_T$. We denote by ${\mathcal D}$
the admissible space-time discretization of $Q_T$ composed of an admissible
mesh ${\mathcal T}$ and the values $(\Delta t,N_T)$. We introduce the space $\Hi_\D$ of piecewise constant in time functions with values in $\Hi_{\T}$,
\begin{multline*}
  \Hi_\D = \bigg\{ w:\overline\Omega\times[0,T]\to\R:
	\exists(w^n)_{n=1,\ldots,N_T}\subset\Hi_{\T}, \quad
	w(x,t) = \sum_{n=1}^{N_T} w^n(x)\mathbf{1}_{(t_{n-1},t_n]}(t)\bigg\},
\end{multline*}
equipped, for $1 \leq p,q < \infty$, with the discrete $L^p(0,T;W^{1,q}(\Omega))$ norm
\begin{align*}
  \bigg(\sum_{n=1}^{N_T}\Delta t \|w^n\|_{1,q,\T}^{p}\bigg)^{1/p}.
\end{align*}

\subsection{Numerical scheme}

We first discretize the initial conditions as
\begin{align}\label{2.ic}
u^0_K = \frac{1}{\m(K)} \int_{K} u^0(x) \, \mathrm{d}x, \quad v^0_K = \frac{1}{\m(K)}\int_{K} v^0(x) \, \mathrm{d}x, \quad \forall K \in \T.
\end{align}
Now for a given $n\geq 1$ and a given $(u^{n-1},v^{n-1}) \in \R^{2\mathrm{Card}(\T)}$, the numerical scheme is defined as
\begin{align}\label{2.schu}
\m(K)\frac{u^n_K-u^{n-1}_K}{\Delta t} &= \sum_{\sigma \in \E_K} \tau_\sigma \, \Dd_{K,\sigma} { \left( u^n \gamma(v^n) \right)}, \quad \forall K \in \T,\\
\eps \, \m(K)\frac{v^n_K-v^{n-1}_K}{\Delta t} &= \delta \sum_{\sigma \in \E_K} \tau_\sigma \, \Dd_{K,\sigma} v^n + \m(K) \, (u^{n-1}_K-\beta v^n_K), \quad \forall K \in \T.\label{2.schv}
\end{align}
We notice that the scheme~\eqref{2.schu}--\eqref{2.schv} is linear and it can be rewritten as
\begin{align}\label{2.schmat}
\Mv \, v^n = b^{n-1}_v, \qquad \Mu \, u^n = b^{n-1}_u,
\end{align}
where $b^{n-1}_v=( \m(K)(\eps\,v^{n-1}+\Delta t u^{n-1}))_{K\in\T}$, $b^{n-1}_u = (\m(K) u^{n-1}_K)_{K \in \T}$ and $\Mv$ and $\Mu$ are squared matrices of size $\mathrm{Card}(\T)\times\mathrm{Card}(\T)$ with entries
\begin{equation}
\left\{
\begin{array}{ll}
(\Mv)_{K,K} =\m(K)(\eps+\Delta t \beta) + \delta \Delta t \sum_{\sigma \in \E_{{\rm int},K}} \tau_\sigma,  &\forall K \in \T,\\
(\Mv)_{K,L} = -\delta\, \Delta t  \, \tau_\sigma,  &\forall K \in \T, \forall L \in \Ne(K), \, \sigma=K|L,	 \\
(\Mv)_{K,L} = 0,  &\forall K \in \T, \forall L \notin \Ne(K),
\end{array}
\right.
\end{equation}
and
\begin{equation}
\left\{
\begin{array}{ll}
(\Mu)_{K,K} =\m(K) + \Delta t \sum_{\sigma \in \E_{{\rm int},K}} \tau_\sigma \gamma(v^n_K),  &\forall K \in \T,\\
(\Mu)_{K,{L}} =- \Delta t  \, \tau_\sigma \gamma(v^n_L),  &\forall K \in \T, \forall L \in \Ne(K), \,\sigma=K|L,\\
(\Mu)_{K,L} =0,  &\forall K \in \T, \forall L \notin \Ne(K).
\end{array}
\right.
\end{equation}

\subsection{Main results}

Our first main result deals with the existence of a unique solution to the scheme~\eqref{2.ic}--\eqref{2.schv} at each time step. But first let us introduce the definition of the discrete entropy functional
\begin{align}\label{2.defent}
H(u^n,v^n) = \sum_{K \in \T} \m(K) \left(h(u^n_K) + \frac{\beta}{2} |v^n_K|^2 - u^n_K v^n_K \right) + \frac{\delta}{2} \sum_{\sigma \in \E} \tau_\sigma \, \left(\Dd_{\sigma} v^n\right)^2,
\end{align}
where $h(x)=x(\log(x)-1)+1$. We also introduce the corresponding discrete entropy dissipation functional as follows
\begin{align}\label{2.defdiss}
D(u^n,v^n,v^{n-1}) = 4\sum_{\sigma \in \E} \tau_\sigma \left(\Dd_{\sigma} \sqrt{u^n \gamma(v^n)}\right)^2 + \eps \sum_{K\in \T} \m(K) \left|\frac{v^n_K-v^{n-1}_K}{\Delta t}\right|^2.
\end{align}

\begin{theorem}[Well-posedness of the scheme]\label{thm.exi}
Assume \emph{\textbf{(H1)}--\textbf{(H4)}}. Then, for any $1\leq n \leq N_T$, there exists a unique positive solution $(u^n,v^n)$ to the scheme~\eqref{2.ic}--\eqref{2.schv}. Moreover this solution satisfies the following properties:
\begin{itemize}
\item Mass estimates: for all $n\in \{1,\ldots,N_T\}$ we have
\begin{align}\label{2.massu}
\langle u^n\rangle  = \langle u^0 \rangle = \frac{1}{\m(\Omega)} \int_\Omega u^0(x) \, \dist x,
\end{align}
and
\begin{align}\label{2.massv}
\langle v^n\rangle = \frac{\langle u^0 \rangle}{\beta}   + \left( \frac{\eps}{\eps+\beta \Delta t}\right)^n \left(\langle v^0\rangle -\frac{\langle u^0 \rangle}{\beta} \right).
\end{align}
\item Entropy dissipation inequality: for all $n\in \{1,\ldots,N_T\}$ it holds
\begin{align}\label{2.EI}
H(u^n,v^n) + \Delta t \, D(u^n,v^n,v^{n-1}) \leq H(u^{n-1},v^{n-1}),
\end{align}
where $H$ and $D$ are defined in~\eqref{2.defent} and~\eqref{2.defdiss}.
\end{itemize}
\end{theorem}

\begin{remark}\label{rk.limit_delta0}
Let us emphasize that the statement of Theorem~\ref{thm.exi} { also} holds in the quasi-stationary regime $\eps=0$. It actually also holds in the asymptotic regime $\delta=0$. Let us precise nonetheless that the regime $\delta=0$ is very degenerate and the convergence of the scheme as well as the well-posedness of the continuous system are far from being { guaranteed} in this case.
\end{remark}

The next main result deals with the convergence of the scheme. We first need to introduce some notation. For $K\in\T$ and $\sigma\in\E_K$,
we define the cell $T_{K,\sigma}$ of the dual mesh:
\begin{itemize}
\item If $\sigma=K|L\in\E_{{\rm int},K}$, then $T_{K,\sigma}$ is that cell 
(``diamond'') whose
vertices are given by $x_K$, $x_L$, and the end points of the edge $\sigma$.
\item If $\sigma\in\E_{{\rm ext},K}$, then $T_{K,\sigma}$ is that cell (``triangle'')
whose vertices are given by $x_K$ and the end points of the edge $\sigma$.
\end{itemize}
The cells $T_{K,\sigma}$ define a partition of $\Omega$. Then, the approximate gradient of $w\in \Hi_\D$ is defined by
\begin{align}\label{def.nab}
  \na^{\mathcal D} w(x,t) = \frac{\m(\sigma)}{\m(T_{K,\sigma})}
	(\mathrm{D}_{K,\sigma} w^k)\, \nu_{K,\sigma}
	\quad\mbox{for }x\in T_{K,\sigma},\ t\in(t_{k-1},t_{k}],
\end{align}
where $\nu_{K,\sigma}$ is the unit vector that is normal to $\sigma$ and points
outwards of $K$.

Finally, we introduce a family $(\mathcal{D}_m)_{m\in\N}$ of admissible space-time 
discretizations of $Q_T$ indexed by the size 
$\eta_m=\max\{\Delta x_m,\Delta t_m\}$ of the mesh, satisfying $\eta_m\to 0$
as $m\to\infty$. We denote by $\T_m$ the corresponding meshes of $\Omega$ and
by $\Delta t_m$ the corresponding time step sizes.  
Finally, we set $\na^m:=\na^{\mathcal{D}_m}$.

\begin{theorem}[Convergence of the scheme]\label{thm.conv}
Assume \emph{\textbf{(H1)}--\textbf{(H4)}} and $\eps>0$. Let $(\mathcal{D}_m)_{m\in\N}$ 
be a family of admissible meshes satisfying~\eqref{2.dd} uniformly in $m\in\N$ and such that $\eta_m\to 0$
as $m\to\infty$. Let $(u_m,v_m)_{m\in\N}$ be a family of finite volume solutions to~\eqref{2.ic}--\eqref{2.schu} obtained in Theorem~\ref{thm.exi}. Then, there exists two nonnegative functions $u \in L^\infty(0,T;L^1(\Omega))$ and $v\in L^\infty(0,T;H^1(\Omega)) \cap H^1(0,T;L^2(\Omega))$ such that, up to a subsequence, as $m\to\infty$ we have
\begin{align*}
    u_m &\rightharpoonup u \quad \mbox{weakly in }L^1(Q_T),\\
    v_m &\to v \quad \mbox{strongly in }L^2(Q_T),\\
    \nabla^m v_m &\rightharpoonup \nabla v \quad  \mbox{weakly in }L^2(Q_T).
\end{align*}
Moreover, $(u,v)$ is a weak solution of~\eqref{1.equ}--\eqref{1.IC} in the sense of Definition~\ref{def.weak}. 
\end{theorem}

\begin{remark}
When comparing this result with the existence result from~\cite{BLT21}, we observe that in~\emph{\textbf{(H4)}} the condition $u^0 \in L^2(\Omega)$ is slightly stronger than the ones imposed in~\cite{BLT21}. This stronger assumption arises when establishing a discretized version of~\eqref{1.dual}, see Remark~\ref{rk:initial_L2} for more details.
\end{remark}

Eventually, we will prove that the scheme~\eqref{2.ic}--\eqref{2.schv} is uniformly, w.r.t.~$\eps$ and the size of the meshes, asymptotic preserving.

\begin{theorem}[AP property]\label{thm.AP}
Assume \emph{\textbf{(H1)}--\textbf{(H5)}}, and let $(\mathcal{D}_m)_{m\in\N}$ 
be a family of admissible meshes satisfying~\eqref{2.dd} uniformly in $m\in\N$ and such that $\eta_m\to 0$
as $m\to\infty$. Consider a sequence $(\eps_m)_{m \in \N}\subset[0,\infty)$ such that $\eps_m \to 0$ as $m \to \infty$ and an associated sequence $(u_m,v_m)_{m\in \N}$ of finite volume solutions to~\eqref{2.ic}--\eqref{2.schv} obtained in Theorem~\ref{thm.exi}. Then, there exists $(u,v)$ such that, as
$(\eps_m,\eta_m)_{m\in\N}$ converges to $(0,0)$, we have, up to a subsequence,
\begin{align*}
    u_m &\rightharpoonup u \quad \mbox{weakly in }L^1(Q_T),\\
    v_m &\to v \quad \mbox{strongly in }L^2(Q_T),\\
    \nabla^m v_m &\rightharpoonup \nabla v \quad \mbox{weakly in }L^2(Q_T),\\
    \eps \pa_t v_m &\rightharpoonup 0 \quad \mbox{weakly in }L^2(Q_T).
\end{align*}
Furthermore, $(u,v)$ is a weak solution to the quasi-stationary system in the sense of Definition~\ref{def.weak}.
\end{theorem}

\begin{remark}\label{rk:byproduct} Note that a by-product of Theorem~\ref{thm.AP} is the convergence (up to a subsequence) of the numerical scheme for $\eps=0$ under assumptions \emph{\textbf{(H1)}}--\emph{\textbf{(H4)}}.
Indeed, taking $(\eps_m)_{m\in \N}=0$ and some initial data $u^0$ in $L^2(\Omega)$, and defining $v^0$ as the solution of \eqref{1.eqv} with $\eps=0$ and $u$ replaced by $u^0$, we have that assumption \emph{\textbf{(H5)}} is automatically verified.
In fact, noting that when $\eps=0$ no initial data for $v$ is needed, see \eqref{1.IC}, assumption \emph{\textbf{(H5)}} becomes irrelevant.
\end{remark}

As a direct consequence of Theorem~\ref{thm.AP} and Proposition~\ref{prop.exiWS} with Remark~\ref{Rk.weakstong} we obtain the

\begin{corollary}\label{cor.AP}
Let the assumptions of Theorem~\ref{thm.AP} hold, with $\di=1$ or $2$. Furthermore, assume that $u^0$ belongs to $L^\infty(\Omega)$. Then, there exists $(u,v)$ such that the convergences in Theorem~\ref{thm.AP} hold for the whole sequence $(u_m,v_m)_{m \in \N}$ (and not only subsequences), and such that $(u,v)$ is a weak-strong solution of the quasi-stationary system in the sense of Definition~\ref{def.WSquasistationary}.
\end{corollary}

\section{Existence proof and entropy diminishing property of the scheme}\label{sec.exiproof}

This section is devoted to the proof of Theorem~\ref{thm.exi}. Our strategy is the following. For $\omega \in (0,1)$ given, let us first define the regularized initial data
\begin{align}\label{def.regIC}
u^{0,\omega}_K := u^0_K + \omega, \quad \forall K \in \T.
\end{align}
Then, starting from $(u^{0,\omega},v^0)$ we establish the existence of a (unique) positive solution $(u^{n,\omega},v^{n,\omega})$ to the scheme~\eqref{2.schu}--\eqref{2.schv}, for which we can show mass and entropy estimates. In order to conclude the proof, it will remain to prove some uniform w.r.t. $\omega$ estimates and to pass to the limit $\omega \to 0$ both in the scheme and in the entropy estimate. Finally, we will use this later inequality to prove that the solutions to the scheme~\eqref{2.ic}--\eqref{2.schv} are positive. First, our results on the scheme with regularized initial data are presented in the following lemma.

\begin{lemma}\label{lem.regIC} Assume \emph{\textbf{(H1)}--\textbf{(H4)}}, and for $\omega\in(0,1)$ let $u^{0,\omega}_K$ be defined by~\eqref{def.regIC}. Then, for any $1\leq n \leq N_T$, there exists a unique positive solution $(u^{n,\omega},v^{n,\omega})$ to the scheme~\eqref{2.schu}--\eqref{2.schv} with initialization $(u^{0,\omega}_K,v^0_K)$. Moreover it satisfies the following properties:
\begin{itemize}
\item Mass estimates: for all $n\in \{1,\ldots,N_T\}$ we have
\begin{align}\label{3.massu}
\langle u^{n,\omega} \rangle = \langle u^{0,\omega} \rangle =  \langle u^0 \rangle + \omega,
\end{align}
and
\begin{align}\label{3.massv}
\langle v^{n,\omega}\rangle = \frac{\langle u^0 \rangle+\omega }{\beta}  +  \left(\dfrac{\eps}{\eps+\beta \, \Delta t}\right)^n \left(\langle v^0\rangle - \frac{\langle u^0 \rangle + \omega}{\beta}\right).
\end{align}
\item Entropy dissipation inequality: for all $n\in \{1,\ldots,N_T\}$ it holds
\begin{multline}\label{3.EIeps}
H(u^{n,\omega},v^{n,\omega}) + \Delta t \sum_{\substack{\sigma \in \E_{{\rm int}}\\\sigma = K|L}} \tau_\sigma \, \Dd_{K,\sigma} (u^{n,\omega} \gamma(v^{n,\omega})) \, \Dd_{K,\sigma} \log(u^{n,\omega} \gamma(v^{n,\omega}))\\ +\eps \,\Delta t \sum_{K \in \T} \m(K) \left|\frac{v^{n,\omega}_K - v^{n-1,\omega}_K}{\Delta t}\right|^2 \leq H(u^{n-1,\omega},v^{n-1,\omega}),
\end{multline}
where $H$ is defined in~\eqref{2.defent}.
\end{itemize}
\end{lemma}

\begin{proof}
We divide the proof into three steps.\\

\noindent{\it Step 1. Well-posedness of the scheme.} In the following we argue by induction. If $n=1$, we notice that all the diagonal entries of $\Mv$ are strictly positive, whereas the extra-diagonal coefficients are nonpositive. Besides, for a given $K \in \T$ it holds
\begin{align*}
\sum_{L \in \T} (\Mv)_{K,L} = \m(K)(\eps+\Delta t \beta) > 0.
\end{align*}
Therefore the matrix $\Mv$ is strictly diagonally dominant with respect to its rows and we conclude that $\Mv$ is a M-matrix. In particular $\Mv$ is invertible and, since all the components of the vectors $b^{0}_v$ are positive, we deduce that there exists a unique vector $v^{1,\omega}$ with positive components satisfying~\eqref{2.schv}. Similarly, one can show that $\mathbb{M}_u^1$ is strictly diagonally dominant with respect to its columns. Since in $\mathbb{M}_u^1$ all the diagonal entries are strictly positive, whereas the extra-diagonal coefficients are nonpositive, we conclude that $\mathbb{M}_u^1$ is also a M-matrix. This implies the existence of a unique positive vector $u^{1,\omega}$ solution to~\eqref{2.schu}. Finally, arguing recursively we obtain the existence of a unique positive solution $(u^{n,\omega},v^{n,\omega})$ to the scheme~\eqref{2.schu}--\eqref{2.schv} for all $1 \leq n \leq N_T$.\\

\noindent{\it Step 2. Mass estimates.} Now, we observe that by summing~\eqref{2.schu} over $K \in \T$ we have by construction
\begin{align*}
\sum_{K \in \T} \m(K) u^{n,\omega}_K = \sum_{K\in \T} \m(K) u^{n-1,\omega}_K,
\end{align*}
such that~\eqref{3.massu} holds.
Similarly, we sum the equation~\eqref{2.schv} over $K \in \T$ and we obtain
\begin{multline*}
\eps \m(\Omega) \langle v^{n,\omega}\rangle - \eps \, \m(\Omega) \frac{\langle u^0 \rangle+\omega}{\beta} = \eps \, \m(\Omega)\, \langle v^{n-1,\omega}\rangle\\ - \eps \, \m(\Omega) \frac{\langle u^0 \rangle+\omega}{\beta} + \Delta t \, \beta \left(\m(\Omega) \frac{\langle u^0 \rangle+\omega}{\beta}-\m(\Omega) \, \langle v^{n,\omega}\rangle \right),
\end{multline*}
so that
\begin{align*}
(\eps+\Delta t \, \beta) \left( \langle v^{n,\omega}\rangle  - \frac{\langle u^0 \rangle+\omega}{\beta} \right) = \eps \langle v^{n-1,\omega}\rangle  - \eps \frac{\langle u^0 \rangle+\omega}{\beta}.
\end{align*}
Hence
\begin{align*}
    \langle v^{n,\omega}\rangle - \frac{\langle u^0 \rangle+\omega}{\beta} =  \dfrac{\eps}{\eps+\beta \, \Delta t} \left(\langle v^{n-1,\omega}\rangle - \frac{\langle u^0 \rangle + \omega}{\beta}\right),
\end{align*}
and we conclude~\eqref{3.massv} by a direct induction.
\medskip

\noindent{\it Step 3. Entropy dissipation inequality.} Let $1\leq n \leq N_T$ be fixed, we multiply~\eqref{2.schu} by $\Delta t \log(u^{n,\omega}_K \gamma(v^{n,\omega}_K)) = \Delta t ( \log(u^{n,\omega}_K) - v^{n,\omega}_K)$, we sum over $K \in \T$ and we use a discrete integration by parts to obtain
\begin{multline*}
\sum_{K \in \T} \m(K) (u^{n,\omega}-u^{n-1,\omega}) \,(\log(u^{n,\omega}_K)-v^{n,\omega}_K) \\+ \Delta t \sum_{\substack{\sigma \in \E_{{\rm int}}\\\sigma = K|L}} \tau_\sigma \, \Dd_{K,\sigma} (u^{n,\omega} \gamma(v^{n,\omega})) \, \Dd_{K,\sigma} \log(u^{n,\omega} \gamma(v^{n,\omega}))=0.
\end{multline*}
Using the convexity of the function $h(x)= x(\log(x)-1)+1$ we get
\begin{multline}\label{3.EI1}
\sum_{K\in \T} \m(K) \left(h(u^{n,\omega}_K)-h(u^{n-1,\omega}_K)\right) - \sum_{K\in \T}\m(K)(u^{n,\omega}_K-u^{n-1,\omega}_K) v^{n,\omega}_K \\+ \Delta t \sum_{\substack{\sigma \in \E_{{\rm int}}\\\sigma = K|L}} \tau_\sigma \, \Dd_{K,\sigma} (u^{n,\omega} \gamma(v^{n,\omega})) \, \Dd_{K,\sigma} \log(u^{n,\omega} \gamma(v^{n,\omega})) \leq 0.
\end{multline}
Now, we multiply the equation~\eqref{2.schv} by $(v^{n,\omega}_K-v^{n-1,\omega}_K)$, we sum over $K \in \T$ and we use a discrete integration by parts to get
\begin{multline*}
\eps\sum_{K \in \T} \m(K) \frac{(v^{n,\omega}_K - v^{n-1,\omega}_K)^2}{\Delta t}  +  \delta \sum_{\substack{\sigma \in \E_{{\rm int}}\\\sigma = K|L}} \tau_\sigma \Dd_{K,\sigma} v^{n,\omega} \, \Dd_{K,\sigma} (v^{n,\omega}-v^{n-1,\omega}) \\+ \beta \sum_{K\in\T} \m(K) v^{n,\omega}_K (v^{n,\omega}_K-v^{n-1,\omega}_K) - \sum_{K\in \T} \m(K) u^{n-1,\omega}_K (v^{n,\omega}_K - v^{n-1,\omega}_K) = 0.
\end{multline*}
Applying the classical inequality $a^2-b^2 \leq 2a(a-b)$ for all $a$ and $b \in \R$ to the second and third terms of the l.h.s. yields
\begin{multline}\label{3.EI2}
\eps\sum_{K \in \T} \m(K) \frac{(v^{n,\omega}_K - v^{n-1,\omega}_K)^2}{\Delta t}  +  \frac{\delta}{2} \sum_{\substack{\sigma \in \E_{{\rm int}}\\\sigma = K|L}} \tau_\sigma \left( (\Dd_\sigma v^{n,\omega})^2 - (\Dd_\sigma v^{n-1,\omega})^2 \right)\\
+ \frac{\beta}{2} \sum_{K\in\T} \m(K) \left(|v^{n,\omega}_K|^2-|v^{n-1,\omega}_K|^2\right) -  \sum_{K\in \T} \m(K) u^{n-1,\omega}_K (v^{n,\omega}_K - v^{n-1,\omega}_K) \leq 0.
\end{multline}
It remains to sum~\eqref{3.EI1} and~\eqref{3.EI2} and we end up with~\eqref{3.EIeps}. 
\end{proof}

It now remains to pass to the limit $\omega \to 0$ and to show the positivity of the limit in order to prove Theorem~\ref{thm.exi}.

\begin{proof}[Proof of Theorem~\ref{thm.exi}]
For $\omega \in (0,1)$, we consider the regularized solution $(u^{n,\omega},v^{n,\omega})$ built in Lemma~\ref{lem.regIC}. Since $\omega \in (0,1)$, we deduce thanks to~\eqref{3.massu} and~\eqref{3.massv}, that there exists a constant $C>0$ independent of $\omega$ (but which depends on $\T$) such that
\begin{align}\label{bound_th4}
0 < u^{n,\omega}_K, \, v^{n,\omega}_K \leq C, \qquad \forall K \in \T, \, n=1,\ldots,N_T.
\end{align}
Thereby, for any $1\leq n \leq N_T$, there exists $(u^n,v^n)$ such that, up to a subsequence, as $\omega \to 0$ we have
\begin{align*}
u^{n,\omega}_K, \, v^{n,\omega}_K \to u^n_K, \, v^n_K \geq 0, \qquad \forall K \in \T.
\end{align*}
We conclude that $(u^n,v^n)$ is a solution to the scheme~\eqref{2.schu}--\eqref{2.schv}. Besides, the linear structure of the scheme allows us to deduce that $(u^n,v^n)$ is the unique solution to the scheme and that in fact, as $\omega \to 0$, the whole sequence $(u^{n,\omega},v^{n,\omega})$ converges toward $(u^n,v^n)$. We now pass to the limit in~\eqref{3.massu} and~\eqref{3.massv} in order to obtain the mass estimates~\eqref{2.massu} and~\eqref{2.massv}. {Moreover, a first consequence of~\eqref{3.EIeps} and the uniform, with respect to $\omega$, bound~\eqref{bound_th4}, is that for all $\sigma \in \E_{{\rm int}}$ with $\sigma = K|L$, we have}
\begin{multline*}
    {\limsup_{\omega\longrightarrow0} \Big\{ (u^{n,\omega}_K\gamma(v^{n,\omega}_K)-u^{n,\omega}_L \gamma(v^{n,\omega}_L)) \, (\log(u^{n,\omega}_K\gamma(v^{n,\omega}_K))-\log(u^{n,\omega}_L \gamma(v^{n,\omega}_L)))\Big\}}\\ {\leq H(u^{n-1},v^{n-1}) + C^2 \, \m(\Omega).}
\end{multline*}
This implies that $u^n$ is positive. Indeed, by contradiction let us assume that there exists $K \in \T$ with $u^n_K=0$. Then, the previous inequality implies that any cell $L \in \Ne(K)$ should satisfy $u^n_L=0$. Then, repeating this argument over all $K \in \T$ we conclude that $u^n_K = 0$ for all $K\in\T$. This contradicts the mass estimate~\eqref{2.massu}. Hence, $u^n$ is a positive vector for all $1\leq n \leq N_T$ and since $\mathbb{M}_v$ is a M-matrix we deduce that $v^n$ is also a positive vector for all $1\leq n \leq N_T$.

We can finally pass to the limit $\omega \to 0$ in~\eqref{3.EIeps} and obtain
\begin{multline*}
H(u^n,v^n) + \Delta t \sum_{\substack{\sigma \in \E_{{\rm int}}\\\sigma = K|L}} \tau_\sigma \, \Dd_{K,\sigma} (u^n \gamma(v^n)) \, \Dd_{K,\sigma} \log(u^n \gamma(v^n))\\ +\eps \,\Delta t \sum_{K \in \T} \m(K) \left|\frac{v^n_K - v^{n-1}_K}{\Delta t}\right|^2 \leq H(u^{n-1},v^{n-1}).
\end{multline*}
We then apply the elementary inequality $4(\sqrt{a}-\sqrt{b})\leq (a-b)(\log(a)-\log(b))$ for all $a$, $b>0$ and we obtain the desired entropy inequality~\eqref{2.EI}.
\end{proof}

\section{Uniform estimates}\label{sec.unifest}

In this section we establish some uniform estimates on the solutions to the scheme~\eqref{2.schu}--\eqref{2.schv}. The main idea, as exposed formally in Section~\ref{sec.keyest}, is to combine the entropy inequality~\eqref{2.EI} together with a duality estimate. Following~\cite[Section 2.1]{BLT21}, we first bound from below the term $-uv$ appearing in the definition of the entropy functional. { Roughly speaking, this is done by estimating the $(H^1)'(\Omega)$ norm of $u-\langle u \rangle$ thanks to the relation \eqref{eq:identity_Hdual}.}
Thereby, we need to adapt this approach at the discrete level. For this purpose, we introduce some discrete counterpart $N$ of the norm $\|\cdot\|_{(H^1)'(\Omega)}$. We first define the subset $\widetilde{\Hi}_\T$ of $\Hi_\T$ as
\begin{align*}
\widetilde{\Hi}_\T = \left\lbrace w \in \Hi_\T \, \, : \, \, \langle w \rangle = 0 \right\rbrace.
\end{align*}
We then define the map $N : \widetilde{\Hi}_\T \to \R_+$ by
\begin{align*}
N(w) = \left(\sum_{\sigma \in \E} \tau_\sigma (\Dd_\sigma z)^2\right)^{1/2} = |z|_{1,2,\T} ,
\end{align*}
where the vector $z$ is the unique solution in $\widetilde{\Hi}_\T$ to
\begin{align}\label{4.aux1}
-\sum_{\sigma \in \E_K} \tau_\sigma \Dd_{K,\sigma} z = \m(K) w_K, \quad \forall K \in \T.
\end{align}
The existence and uniqueness of such $z$ are guaranteed by~\cite[Theorem 2.5]{ChDr11}. One can easily check that $N$ is a norm on $\widetilde{\Hi}_\T$. More precisely, we have the

\begin{lemma}\label{lem.dual}
Let $\T$ be an admissible mesh of $\Omega$. Then, for all $w\in \widetilde{\Hi}_\T$, we have
\begin{align*}
N(w)=
 \sup\bigg\{\int_\Omega w \theta\, \dist x: \theta \in\widetilde{\Hi}_\T,\
	|\theta|_{1,2,\T}=1\bigg\}.
\end{align*}
\end{lemma}

\begin{proof}
Let $w\in \widetilde{\Hi}_\T$ and $\theta \in \widetilde{\Hi}_\T$. We multiply the equation~\eqref{4.aux1} by $\theta_K$, we sum over $K \in \T$, we apply a discrete integration by parts and Cauchy-Schwarz inequality to obtain
\begin{align*}
\int_\Omega w \theta \, \dist x = \sum_{K \in \T} \m(K) \, w_K \, \theta_K = \sum_{\substack{\sigma \in \E_{{\rm int}}\\\sigma=K|L}} \tau_\sigma \Dd_{K,\sigma} z \, \Dd_{K,\sigma} \theta \le N(w)\, |\theta|_{1,2,\T}.
\end{align*}
Then, picking $\theta_K=z_K/N(w)$ for all $K\in\T$, we have that $\theta \in \widetilde{\Hi}_\T$, $|\theta|_{1,2,\T}=1$ and
\begin{align*}
\int_\Omega w \theta \, \dist x = \frac{1}{N(w)} \sum_{\sigma \in \E} \tau_\sigma (\Dd_\sigma z)^2 = N(w).
\end{align*}
This concludes the proof of Lemma~\ref{lem.dual}.
\end{proof}

Equipped with this norm, we now prove the following estimate, which can be seen as the discrete counterpart of~\eqref{1.dual}.

\begin{proposition}[Duality estimate]\label{prop.dual}
Assume \emph{\textbf{(H1)}--\textbf{(H4)}} and let $\D$ be an admissible mesh satisfying~\eqref{2.dd}. Then, there exists a constant $C>0$ only depending on $\Omega$ and $\zeta$ such that
\begin{equation}\label{4.estdual}\begin{split}
\max_{1\leq n \leq N_T} N(u^n-\langle u^0 \rangle)^2 & + 2\sum_{n=1}^{N_T} \Delta t \, \|u^n\sqrt{\gamma(v^n)}\|^2_{L^2(\Omega)} \\
& \leq C \|u^0-\langle u^0 \rangle\|^2_{L^2(\Omega)} + 2\, T \, \m(\Omega) \, \langle u^0 \rangle^2.
\end{split}\end{equation}

\end{proposition}

\begin{proof}
We consider the term
\begin{align*}
I = \frac12\sum_{K \in \T} {\m(K)} \Big((u^n_K-\langle u^0 \rangle)z^n_k-(u^{n-1}_K-\langle u^0 \rangle)z^{n-1}_K\Big),
\end{align*}
where the vectors $z^n$ and $z^{n-1}$ are the unique solutions to~\eqref{4.aux1} associated to $u^n-\langle u^0 \rangle$ and $u^{n-1}-\langle u^0 \rangle$ respectively. Thanks to~\eqref{4.aux1} and a summation by parts, we first observe that we can rewrite $I$ as
\begin{align}\label{4.aux2}
I &= \frac12 \sum_{\substack{\sigma \in \E_{{\rm int}}\\\sigma = K|L}} \tau_\sigma \Big( (\Dd_\sigma z^n)^2 - (\Dd_\sigma z^{n-1})^2 \Big) = \frac12\left[ N(u^n-\langle u^0 \rangle)^2 - N(u^{n-1}-\langle u^0 \rangle)^2\right].
\end{align}
On the other hand, applying the inequality $(a^2-b^2)\leq 2 a(a-b)$ for $a$, $b \in \R$, we obtain
\begin{align*}
I = \frac12 \sum_{\substack{\sigma \in \E_{{\rm int}}\\\sigma = K|L}} \tau_\sigma \Big( (\Dd_\sigma z^n)^2 - (\Dd_\sigma z^{n-1})^2 \Big) &\leq \sum_{\substack{\sigma \in \E_{{\rm int}}\\\sigma = K|L}}  \tau_\sigma ( \Dd_{K,\sigma} z^n - \Dd_{K,\sigma} z^{n-1} ) \, \Dd_{K,\sigma} z^n\\
&= -\sum_{K \in \T} \sum_{\sigma \in \E_K} \tau_\sigma \,  ( \Dd_{K,\sigma} z^n - \Dd_{K,\sigma} z^{n-1} ) \, z^n_K.
\end{align*}
Using the equations~\eqref{4.aux1} and~\eqref{2.schu} and some discrete integration by parts we get
\begin{align*}
I \leq \sum_{K \in \T} \m(K) \, \left((u^n_K-\langle u^0 \rangle)-(u^{n-1}_K-\langle u^0 \rangle)\right) \, z^n_K
&= - \Delta t\sum_{\substack{\sigma \in \E_{{\rm int}}\\\sigma = K|L}} \tau_\sigma \, \Dd_{K,\sigma} \left(u^n \gamma(v^n)\right) \, \Dd_{K,\sigma} z^n\\
&= \Delta t \sum_{K \in \T} \sum_{\sigma \in \E_K} \tau_\sigma \, \Dd_{K,\sigma} z^n \, u^n_K \gamma(v^n_K).
\end{align*}
It remains to apply once more~\eqref{4.aux1} and the inequality $\gamma(v^n_K) \leq 1$ for all $K \in \T$ together with~\eqref{2.massu} and we end up with
\begin{align}
I \leq -\Delta t \sum_{K\in\T} \m(K)\left(u^n_K-\langle u^0 \rangle\right) \, u^n_K \gamma(v^n_K) =& -\Delta t \, \|u^n\sqrt{\gamma(v^n)}\|^2_{L^2(\Omega)}  + \Delta t \, \langle u^0 \rangle \|u^n \gamma(v^n)\|_{L^1(\Omega)} \nonumber \\
\leq&  -\Delta t \, \|u^n\sqrt{\gamma(v^n)}\|^2_{L^2(\Omega)}  + \Delta t \, \m(\Omega) \langle u^0 \rangle^2. \label{4.aux3}
\end{align}
Gathering~\eqref{4.aux2} and~\eqref{4.aux3} we obtain
\begin{align*}
N(u^n-\langle u^0 \rangle)^2 + 2 \Delta t \, \|u^n\sqrt{\gamma(v^n)}\|^2_{L^2(\Omega)} \leq N(u^{n-1}-\langle u^0 \rangle)^2 + 2 \Delta t \, \m(\Omega) \langle u^0 \rangle^2.
\end{align*}
We sum over $n$ the above inequality, such that, it holds:
\begin{align*}
\max_{1 \leq n \leq N_T} N(u^n-\langle u^0 \rangle)^2 + 2 \sum_{n=1}^{N_T} \Delta t \, \|u^n\sqrt{\gamma(v^n)}\|^2_{L^2(\Omega)} \leq N(u^0-\langle u^0 \rangle)^2 + 2 \, T \, \m(\Omega)\, \langle u^0 \rangle^2.
\end{align*}
We now apply the Cauchy-Schwarz inequality, Lemma~\ref{lem.con} in Appendix~\ref{app.A} and the discrete Poincar\'e-Wirtinger inequality (see~\cite[Proposition 1]{BCF15}) and we have
\begin{align}
N(u^0-\langle u^0 \rangle)^2 &=|z^0|_{1,2,\T}^2 = \sum_{K \in \T} \m(K) (u^0_K-\langle u^0 \rangle) \, z^0_K \nonumber\\
&\leq \left(\sum_{K\in\T} \m(K) \,\left|u^0_K-\langle u^0 \rangle\right|^2\right)^{1/2} \, \|z^0\|_{L^2(\Omega)} \leq C \|u^0-\langle u^0 \rangle\|_{L^2(\Omega)} \, |z^0|_{1,2,\T}, \label{4.aux4}
\end{align}
for some $C$ depending only on $\Omega$ and $\zeta$.
Thereby, combining \eqref{4.aux4} with the previous inequality 
we deduce~\eqref{4.estdual}.
\end{proof}

\begin{remark}\label{rk:initial_L2}
The counterpart of Proposition~\ref{prop.dual} in a continuous setting, given by the equation~\eqref{1.dual}, only requires the $(H^1)'$ norm of the initial data (minus its mean) while here the obtained inequality depends on its $L^2$ norm. To discuss this divergence, let us first introduce the zero-mean $L^2$ set
\begin{equation*}
   L_{\Diamond}^2(\Omega):=\left\{ w \in L^2(\Omega): \langle w\rangle =0\right\}.
\end{equation*}
In the proof of Proposition~\ref{prop.dual}, the $L^2$ norm of the initial data arises in~\eqref{4.aux4}, where we obtain the continuity of the discrete projection given by~\eqref{2.ic} and~\eqref{def:H_T},
\begin{equation}\label{4.proj}
    \pi_\T: w \in L_{\Diamond}^2(\Omega) \mapsto \sum_{K\in\T} \mathbf{1}_K \frac{1}{\m(K)} \int_{K} w\, \mathrm{d}x \in \widetilde{\Hi}_\T,
\end{equation}
when $\widetilde{\Hi}_\T$ is endowed with the norm $N$, and $L_{\Diamond}^2(\Omega)$ is naturally endowed with the $L^2$ norm. However, one cannot hope in general to obtain the continuity of the discrete projection $\pi_\T$ when extended and seen as an operator from $((H^1)'(\Omega),\|\cdot\|_{(H^1)'(\Omega)})$ to $(\widetilde{\Hi}_\T,N)$, as showed in the counter-example given in Appendix~\ref{sec.countex}.
Actually, while $N$ seen as a dual norm (see Lemma~\ref{lem.dual}) and the norm $\|\cdot\|_{(H^1)'(\Omega)}$ share formal similarities, they are quite delicate to compare. One reason is that the intersection of the dual sets $\widetilde{\Hi}_\T$ and $H^1(\Omega)$ is trivial.

We could however slightly weaken the assumption $u^0 \in L^2(\Omega)$ by $u^0 \in L^p(\Omega)$ for some convenient $p<2$. Indeed, it is sufficient in the estimate~\eqref{4.aux4} to apply H\"older's inequality as follows
\begin{align*}
    \sum_{K\in\T} \m(K) \, (u^0_K - \langle u^0 \rangle) \, z^0_K \leq \|u^0-\langle u^0 \rangle \|_{L^p(\Omega)} \, \|z^0\|_{L^q(\Omega)},
\end{align*}
with $1/p+1/q=1$. Then, depending on the space dimension $\di=1$, $2$ or $3$, one can readily obtain the convenient exponent $p<2$.
\end{remark}

\begin{corollary}\label{coro.unif1}
Assume \emph{\textbf{(H1)}--\textbf{(H4)}} and let $\D$ be an admissible mesh satisfying~\eqref{2.dd}. Then, there exists a constant $C>0$ which depends on $u^0$, $v^0$, $\beta$, $\delta$, $T$, $\Omega$ and $\zeta$ such that
\begin{multline}\label{4.unif}
\max_{1\leq n\leq N_T}\sum_{K \in \T} \m(K) h(u^n_K) + 4 \sum_{n=1}^{N_T} \Delta t \, |\sqrt{u^n \, \gamma(v^n)}|^2_{1,2,\T}\\ + \max_{1\leq n \leq N_T} \, \|v^n\|_{1,2,\T} +\eps \sum_{n=1}^{N_T} \Delta t \left\|\frac{v^n-v^{n-1}}{\Delta t}\right\|^2_{L^2(\Omega)} \leq C.
\end{multline}
\end{corollary}

\begin{proof}
As already mentioned, in order to establish~\eqref{4.unif}, the main idea is to use the entropy inequality~\eqref{2.EI}. For this purpose we need to bound from below the term $-u^n_K v^n_K$ appearing in the definition~\eqref{2.defent} of the functional $H$. In the sequel we denote by $z^n$ the unique solution to~\eqref{4.aux1} associated to $u^n-\langle u^0\rangle$. We have the following relations
\begin{align*}
\sum_{K\in \T} \m(K) \, u^n_K v^n_K &= \sum_{K\in \T} \m(K) (u^n_K -\langle u^0 \rangle) v^n_K + \sum_{K\in\T} \m(K) \, \langle u^0 \rangle v^n_K\\
&= -\sum_{K\in\T}  \sum_{\sigma \in \E_K} \tau_\sigma \,\Dd_{K,\sigma} z^n \, v^n_K + \langle u^0 \rangle \sum_{K\in\T} \m(K) \, v^n_K\\
&= \sum_{\substack{\sigma \in \E_{{\rm int}}\\\sigma = K|L}} \tau_\sigma \, \Dd_{K,\sigma} z^n \, \Dd_{K,\sigma} v^n + \langle u^0 \rangle \sum_{K\in\T} \m(K) \,  v^n_K.
\end{align*}
Applying Young's inequality we obtain
\begin{align*}
-\sum_{K\in \T} \m(K) \, u^n_K v^n_K \geq -\frac{N(u^n-\langle u^0 \rangle)^2}{\delta} - \frac{\delta}{4} \, |v^n|^2_{1,2,\T} - \frac{\langle u^0 \rangle^2}{\beta} - \frac{\beta}{4} \, \|v^n\|^2_{L^2(\Omega)}.
\end{align*}
Therefore, thanks to Proposition~\ref{prop.dual} we deduce that there exists a constant $C>0$ only depending on $u^0$, $T$, $\Omega$ and $\zeta$ such that
\begin{align*}
H(u^n,v^n) \geq -\frac{C}{\delta} - \frac{\langle u^0 \rangle^2}{\beta} + \sum_{K\in\T} \m(K) \left(h(u^n_K) + \frac{\beta}{4} |v^n_K|^2 \right) + \frac{\delta}{4} \sum_{\sigma \in \E} \tau_\sigma \, (\Dd_\sigma v^n)^2.
\end{align*}
Hence, the inequality~\eqref{4.unif} is a simple consequence of the above estimate, the entropy production inequality~\eqref{2.EI}, Lemma~\ref{lem.con} and Lemma~\ref{lem.H1con}.
\end{proof}

\begin{proposition}\label{prop.timetranslate}
Assume \emph{\textbf{(H1)}--\textbf{(H4)}} with $\eps>0$ and let $\D$ be an admissible mesh satisfying~\eqref{2.dd}. Then there exists a constant $C>0$ only depending on $u^0$, $v^0$, $\beta$, $\eps$, $\delta$, $T$, $\Omega$ and $\zeta$, which blows-up as $\eps \to 0$, such that
\begin{align}\label{5.timetranslate}
    \|v(\cdot, \cdot+\tau)-v(\cdot,\cdot)\|^2_{L^2(\Omega \times (0,T-\tau)} \leq C \, \tau, \qquad \tau\in(0,\Delta t),
\end{align}
where $v \in \Hi_\D$ denotes the piecewise constant in time function associated to the solutions $v^n$ of~\eqref{2.schv} for $1 \leq n \leq N_T$.
\end{proposition}

\begin{proof}
Following, for instance~\cite{EGH00}, we first notice that it holds
\begin{align*}
    \int_0^{T-\tau} \int_\Omega \left(v(x,t+\tau)-v(x,t)\right)^2 \, \dist x \dist t \leq \sum_{n=1}^{N_T} \int_0^{T-\tau} \chi_n(t,t+\tau) \, \dist t \, \sum_{K\in \T} \m(K) \, (v^n_K-v^{n-1}_K)^2,
\end{align*}
where
\begin{align*}
    \chi_n(t,t+\tau) =
    \left\{
     \begin{array}{ll}
        1 & \mbox{if } n\Delta t \in (t,t+\tau], \\
        0 & \mbox{if } n\Delta t \notin (t,t+\tau]. 
    \end{array}
\right.
\end{align*}
Let us observe that by construction we have
\begin{align*}
    \int_0^{T-\tau} \chi_n(t,t+\tau) \, \dist t  \leq \tau.
\end{align*}
Thus
\begin{align*}
     \int_0^{T-\tau} \int_\Omega \left(v(x,t+\tau)-v(x,t)\right)^2 \, \dist x\dist t \leq \tau \sum_{n=1}^{N_T} \sum_{K\in \T} \m(K) \, (v^n_K-v^{n-1}_K)^2.
\end{align*}
Thanks to Corollary~\ref{coro.unif1} we deduce that there exists a constant $C>0$ only depending on $u^0$, $v^0$, $\beta$, $\eps$, $\delta$, $T$, $\Omega$ and $\zeta$ such that~\eqref{5.timetranslate} holds.
\end{proof}

\section{Proof of convergence of the scheme}\label{sec.convproof}

\subsection{Compactness properties}

Let the discrete time derivative of a function $w\in \Hi_{\D_m}$ be given by
\begin{align*}
    \partial_t^m w(x,t) = \frac{w^n(x)-w^{n-1}(x)}{\Delta t_m} \quad \mbox{for } (x,t) \in \Omega \times (t_{n-1},t_n].
\end{align*}

\begin{proposition}\label{prop.compacite}
Assume \emph{\textbf{(H1)}--\textbf{(H4)}} and $\eps>0$. Let $(\mathcal{D}_m)_{m\in\N}$ 
be a family of admissible meshes satisfying~\eqref{2.dd} uniformly in $m\in\N$. Let $(u_m,v_m)_{m\in\N}$ be a family of finite volume solutions to~\eqref{2.ic}--\eqref{2.schu} obtained in Theorem~\ref{thm.exi}. Then, there exists $u \in L^\infty(0,T;L^1(\Omega))$ and $v\in L^\infty(0,T;H^1(\Omega)) \cap H^1(0,T;L^2(\Omega))$ two nonnegative functions such that, up to a subsequence, as $m\to\infty$ we have
\begin{align}
u_m &\rightharpoonup u \quad \mbox{weakly in }L^1(Q_T),\label{5.convu}\\
v_m &\to v \quad \mbox{strongly in }L^2(Q_T),\label{5.convv}\\
\nabla^m v_m &\rightharpoonup \nabla v \quad \mbox{weakly in }L^2(Q_T),\label{5.convnabv}\\
\partial_t^m v_m &\rightharpoonup \partial_t v \quad \mbox{weakly in }L^2(Q_T),\label{5.convdt}\\
u_m \, \gamma(v_m) &\rightharpoonup u \, \gamma(v) \quad \mbox{weakly in }L^2(Q_T).\label{5.convugam}
\end{align}
\end{proposition}

\begin{proof}
All the convergence properties stated below occur up to the extraction of a subsequence when $m \to \infty$. The existence of a non-negative function $u\in L^1(Q_T)$ such that the convergence result~\eqref{5.convu} holds, is a direct consequence of~\eqref{4.unif} and the Dunford-Pettis theorem. {The weak  $L^1(Q_T)$ convergence and the mass estimate~\eqref{2.massu} show that $u \in L^\infty(0,T;L^1(\Omega))$ and that mass is conserved at the limit.} 
Now, the existence of a function $v \in L^2(Q_T)$ satisfying~\eqref{5.convv} is a consequence of Corollary~\ref{coro.unif1} and Proposition~\ref{prop.timetranslate} in conjunction with a Kolmogorov compactness theorem in $L^2$, see~\cite[Theorem 3.9, Theorem 4.2]{EGH00}. This convergence implies that $v_m \to v$ a.e. in $Q_T$, so that $v$ is a nonnegative function. Moreover, Corollary~\ref{coro.unif1} yields the existence of a function $w$ such that $\nabla^m v_m \rightharpoonup w$ weakly in $L^2(Q_T)$. Then, we readily show that in the sense of distribution it holds $w=\nabla v$. The fact that $v \in L^\infty(0,T;H^1(\Omega))$ is a consequence of the entropy inequality~\eqref{2.EI} and the l.s.c. of the $L^2$ norm. Applying once more Corollary~\ref{coro.unif1} allows us to observe that the sequence $(\partial_t^m v_m)_{m\in\N}$ is uniformly bounded in $L^2(Q_T)$, which yields the convergence result~\eqref{5.convdt}. Finally, for~\eqref{5.convugam}, we first notice that $u_m \rightharpoonup u$ weakly in $L^1(Q_T)$. Besides, since the sequence of nonnegative functions $(v_m)_{m\in\N}$ 
converge a.e in $Q_T$
and since the function $\gamma$ is continuous then
\begin{align*}
\gamma(v_m) \to \gamma(v) \quad \mbox{a.e. in }Q_T
\quad \mbox{and} \quad \sup_{m \in \N} \|\gamma(v_m)\|_{L^\infty(Q_T)} \leq 1.
\end{align*}
Hence, thanks to Lemma~\ref{lem.convugamm} we conclude that $u_m \gamma(v_m)  \, \rightharpoonup\,  u \gamma(v)$ weakly in $L^1(Q_T)$. Finally, using again the uniform bound on $\gamma(v_m)$ we deduce from Proposition~\ref{prop.dual} that the sequence $(u_m \gamma(v_m))_{m \in \N}$ is uniformly bounded in $L^2(Q_T)$. We can then improve its weak convergence and obtain~\eqref{5.convugam}. 
\end{proof}

\subsection{Identification of the limit}\label{sec.limitscheme}
We can finally prove the convergence of the scheme.

\begin{proof}[Proof of Theorem~\ref{thm.conv}]

It remains to show that the functions $u$ and $v$ obtained in Proposition~\ref{prop.compacite} are weak solutions to~\eqref{1.equ}--\eqref{1.IC} in the sense of Definition~\ref{def.weak}. In order to prove that $v$ satisfies~\eqref{1.weakv} we refer for instance to~\cite[Theorem 4.2]{EGH00}. Now, let $\varphi \in C^{\infty}_c(\overline{\Omega} \times [0,T))$ with $\nabla \varphi \cdot \nu = 0$ on $\partial \Omega \times [0,T)$, we multiply~\eqref{2.schu} by $\Delta t_m \varphi_K^{n-1}$ where $\varphi_K^{n-1} = \varphi(x_K,t_{n-1})$, we sum over $K \in \T_m$ and $1\leq n \leq N_T$, and we obtain $F_1^m+F_2^m = 0$ where
\begin{align*}
F_1^m = \sum_{n=1}^{N_T} \sum_{K \in \T_m} \m(K) (u^n_K - u^{n-1}_K) \, \varphi^{n-1}_K = \sum_{n=1}^{N_T} \sum_{K\in \T_m} \m(K) u^n_K ( \varphi^{n-1}_K - \varphi^n_K) - \sum_{K\in\T_m} \m(K) \, u^0_K \varphi^0_K,
\end{align*}
and
\begin{align*}
F_2^m = -\sum_{n=1}^{N_T} \Delta t_m \sum_{K\in\T_m}\sum_{\sigma \in \E_K} \tau_\sigma \, (\Dd_{K,\sigma} u^n \gamma(v^n)) \, \varphi^{n-1}_K.
\end{align*} 
Let $\psi_m(x,t) = (\varphi^{n}_K - \varphi^{n-1}_K)/\Delta t_m$ for all $x\in K$, $K\in \T_m$ and $t\in(t_{n-1}, t_n]$ and $\varphi_m^0(x) = \varphi^{0}_K$ for all $x\in K$. Then, since $\psi_m \to\partial_t\varphi $ in $L^\infty(Q_T)$ and $\varphi^0_m\to \varphi(\cdot,0)$ in $L^\infty(\Omega)$, using the convergence results of Proposition~\ref{prop.compacite} one obtains that, up to a subsequence,
\begin{align*}
&F^m_1 +\int_{Q_T} u \partial_t \varphi \, \dist x \dist t + \int_{\Omega} u^0(x) \varphi(x,0) \, \mathrm{d}x\\
=& \int_{Q_T} (u \partial_t \varphi - u_m \psi_m)\, \dist x \dist t +  \int_{\Omega} (u^0(x) \varphi(x,0) - u^0_m(x)\varphi_m^0(x))\, \dist x \to0\quad \mbox{as }m\to \infty.
\end{align*}
Now, we rewrite $F_2^m$ as
\begin{align*}
    F_2^m = -\frac12 \sum_{n=1}^{N_T} \Delta t_m \sum_{\sigma \in \E} \m(\sigma) \, \left(u^n_{K,\sigma} \gamma(v^n_{K,\sigma}) - u^n_K \gamma(v^n_K) \right) \, \frac{\varphi^{n-1}_K-\varphi^{n-1}_{K,\sigma}}{\dist_\sigma}.
\end{align*}
Let us now introduce the following quantity
\begin{align*}
    F^m_{21} &= \sum_{n=1}^{N_T} \Delta t_m \int_\Omega u_m \gamma(v_m) \, \Delta \varphi(x,t_{n-1}) \, \dist x\\
    &= \sum_{n=1}^{N_T} \Delta t_m \sum_{K \in \T_m} u^n_K \gamma(v^n_K) \, \int_K \Delta \varphi(x,t_{n-1}) \, \dist x\\
    &= \frac12 \sum_{n=1}^{N_T} \Delta t_m \sum_{\sigma \in \E} \left(u^n_K \gamma(v^n_K) - u^{n}_{K,\sigma} \gamma(v^n_{K,\sigma})\right) \, \int_{\sigma} \nabla \varphi(s,t_{n-1}) \cdot \nu_{K,\sigma} \, \dist s. 
\end{align*}
Thanks to the weak convergence, up to a subsequence, in $L^2(Q_T)$ of $(u_m \, \gamma(v_m))_{m \in \N}$ towards $u\gamma(v)$ we obtain
\begin{align*}
F^m_{21}  \to \int_{Q_T} u \gamma(v) \, \Delta \varphi \, \dist x\dist t \quad \mbox{as }m\to\infty.
\end{align*}
Besides, we notice that the term $F^m_2+F^m_{21}$ can be rewritten as
\begin{align*}
    F^m_2+F^m_{21} = \frac12 \sum_{n=1}^{N_T} \Delta t_m \sum_{\sigma \in \E} \m(\sigma) \left(u^n_K\gamma(v^n_K)-u^n_{K,\sigma} \gamma(v^n_{K,\sigma})\right) \, R^{n-1}_\sigma,
\end{align*}
where
\begin{align*}
    R^{n-1}_\sigma := \frac{1}{\m(\sigma)} \int_{\sigma} \nabla \varphi(s,t_{n-1}) \cdot \nu_{K,\sigma} \,\dist s - \frac{\varphi^{n-1}_{K,\sigma}-\varphi^{n-1}_K}{\dist_\sigma}.
\end{align*}
Using the regularity of $\varphi$ we deduce that there exists a constant $C> 0$ only depending on $\varphi$ such that
\begin{align*}
    \left|R^{n-1}_\sigma\right| \leq C \, \eta_m.
\end{align*}
This yields
\begin{multline*}
    \left|F^m_2+F^m_{21}\right| \leq C \eta_m \sum_{n=1}^{N_T} \Delta t_m \sum_{\substack{\sigma \in \E_{{\rm int}}\\\sigma=K|L }} \m(\sigma) \Bigg| \left(\sqrt{u^n_K \gamma(v^n_K)}+\sqrt{u^n_{L} \gamma(v^n_{L})} \right) \\ \times \left(\sqrt{u^n_K \gamma(v^n_K)}-\sqrt{u^n_{L} \gamma(v^n_{L})} \right)\Bigg|.
\end{multline*}
Now, thanks to Cauchy-Schwarz inequality and the discrete entropy estimate~\eqref{2.EI} we obtain the existence of a constant, still denoted $C$, independent of $\eta_m$ such that
\begin{align*}
    \left|F^m_2+F^m_{21}\right| \leq C \eta_m \left(\sum_{n=1}^{N_T} \Delta t_m \sum_{\substack{\sigma \in \E_{{\rm int}}\\\sigma=K|L }} \m(\sigma) \dist_\sigma  \left(\sqrt{u^n_K \gamma(v^n_K)}+\sqrt{u^n_{L} \gamma(v^n_{L})} \right)^2\right)^{1/2}.
\end{align*}
We estimate the remaining term as follows
\begin{align*}
    \sum_{n=1}^{N_T} \Delta t_m \sum_{\substack{\sigma \in \E_{{\rm int}}\\\sigma=K|L }} \m(\sigma) \dist_\sigma  &\Bigg(\sqrt{u^n_K \gamma(v^n_K)}+\sqrt{u^n_{L} \gamma(v^n_{L})} \Bigg)^2 \\ &\leq 2 \sum_{n=1}^{N_T} \Delta t_m \sum_{\substack{\sigma \in \E_{{\rm int}}\\\sigma=K|L }} \m(\sigma) \dist(x_K,x_L) \, \left( u^n_K \gamma(v^n_K)+ u^n_{L} \gamma(v^n_{L}) \right)\\
    &\leq 2\sum_{n=1}^{N_T} \Delta t_m \sum_{K \in \T_m} u^n_K \gamma(v^n_K) \sum_{\sigma \in \E_{{\rm int},K}} \m(\sigma) \, \dist(x_K,x_L),
\end{align*}
where, using furthermore the regularity assumption~\eqref{2.dd}, we observe that in two space dimensions we have
\begin{align*}
    \sum_{\sigma \in \E_{{\rm int},K}} \m(\sigma) \, \dist(x_K,x_L) \le \frac{1}{\zeta} \sum_{\sigma \in \E_{{\rm int},K}} \m(\sigma) \, \dist(x_K,\sigma) \leq \frac{2}{\zeta} \m(K),
\end{align*}
so that
\begin{align*}
    \sum_{n=1}^{N_T} \Delta t_m \sum_{\substack{\sigma \in \E_{{\rm int}}\\\sigma=K|L }} \m(\sigma) \dist_\sigma  \Bigg(\sqrt{u^n_K \gamma(v^n_K)}+&\sqrt{u^n_{L} \gamma(v^n_{L})} \Bigg)^2 \leq \frac{4}{\zeta} \sum_{n=1}^{N_T} \Delta t_m \sum_{K \in \T_m} \m(K) \, u^n_K \gamma(v^n_K).
\end{align*}
Applying once more the Cauchy-Schwarz inequality and the estimate $\gamma(v^n_K)\leq 1$ for all $K \in \T_m,$ we obtain
\begin{multline*}
     \sum_{n=1}^{N_T} \Delta t_m \sum_{\substack{\sigma \in \E_{{\rm int}}\\ \sigma=K|L }} \m(\sigma) \dist_\sigma  \Bigg(\sqrt{u^n_K \gamma(v^n_K)}+\sqrt{u^n_{L} \gamma(v^n_{L})} \Bigg)^2\\ \leq \frac{4 \sqrt{T \, \m(\Omega)}}{\zeta} \left(\sum_{n=1}^{N_T} \Delta t_m \|u_m \sqrt{\gamma(v_m)}\|^2_{L^2(\Omega)} \right)^{1/2}.
\end{multline*}
We end up with the following estimate
\begin{align*}
    |F^m_2+F^m_{21}| \leq \frac{2 C (T \m(\Omega))^{1/4}}{\zeta^{1/2}} \, \left(\sum_{n=1}^{N_T} \Delta t_m \|u_m \sqrt{\gamma(v_m)}\|^2_{L^2(\Omega)} \right)^{1/4} \, \eta_m.
\end{align*}
Finally, the discrete duality estimate~\eqref{4.estdual} allows us to conclude that
\begin{align*}
    F^m_2 + F^m_{21} \to 0 \quad \mbox{as }m\to \infty.
\end{align*}
This finishes the proof of Theorem~\ref{thm.conv}.
\end{proof}

\section{Asymptotic preserving property of the scheme}\label{sec.AP}

This section is dedicated to the proof of Theorem~\ref{thm.AP}. For this purpose we will first adapt at the discrete level the estimate~\eqref{estim.dtv}. We define for every $0 \leq n\leq N_T-1$,
\begin{align}\label{6.defw}
    w^n_K = \frac{v^{n+1}_K-v^n_K}{\Delta t}, \quad \forall K \in \T.
\end{align}
Then, using the equation~\eqref{2.schv}, we deduce that $w^n_K$ satisfies for all $1 \leq n \leq N_T-1$,
\begin{align}\label{6.schw}
\quad \eps \, \m(K) \frac{w^n_K-w^{n-1}_K}{\Delta t} + \beta \m(K) w^n_K = \delta \sum_{K \in \E_K} \tau_\sigma \Dd_{K,\sigma} w^n + \m(K) \frac{u^n_K-u^{n-1}_K}{\Delta t}, \quad \forall K \in \T.
\end{align}
Throughout this section, for each $0\leq n \leq N_T-1$, we introduce the unique element $z^n \in \widetilde{\Hi}_\T$ such that
\begin{align*}
    -\sum_{\sigma \in \E_K} \tau_\sigma \, \Dd_{K,\sigma} z^n = \m(K) (w^n_K-\langle w^n \rangle), \quad \forall K \in \T.
\end{align*}

\begin{proposition}[Discrete counterpart of~\eqref{estim.dtv}] \label{prop.dtv}
Assume \emph{\textbf{(H1)}--\textbf{(H4)}}. Then, for each $1 \leq n \leq N_T-1$, the following estimate hold
\begin{multline}\label{6.unifdtv}
    \frac{\delta}{2} \sum_{K\in \T} \m(K) \, \left(w^n_K-\langle w^n\rangle\right)^2 + \frac{\eps}{2} \sum_{\sigma \in \E} \tau_\sigma \frac{\left(\Dd_\sigma z^n\right)^2-\left(\Dd_\sigma z^{n-1}\right)^2}{\Delta t}\\ + \beta \sum_{\sigma \in \E} \tau_\sigma \left(\Dd_\sigma z^n \right)^2 \leq \frac{1}{2\delta} \sum_{K\in\T} \m(K)  \left(u^n_K \gamma(v^n_K)-\langle u^n \gamma(v^n)\rangle\right)^2.
\end{multline}
\end{proposition}

\begin{proof}
First, we rewrite for every $K \in \T$ the equation~\eqref{6.schw} as
\begin{align*}\label{App.schw}
\eps\, \m(K) \frac{w^n_K-w^{n-1}_K}{\Delta t} &+ \beta \m(K) w^n_K\\
&= \delta \sum_{ \sigma \in \E_K} \tau_\sigma \Dd_{K,\sigma} (w^n-\langle w^n \rangle) + \sum_{\sigma \in \E_K} \tau_\sigma \Dd_{K,\sigma}(u^n \gamma(v^n)-\langle u^n \gamma(v^n)\rangle),
\end{align*}
multiplying this relation by $z^n_K$ and summing over $K\in \T$ we obtain the equation $I_1=I_2$ where
\begin{align*}
    I_1 &=\eps \sum_{K \in \T} \m(K) \frac{w^n_K-w^{n-1}_K}{\Delta t} \, z^n_K + \beta \sum_{K\in\T} \m(K) w^n_K \, z^n_K,\\
    I_2 &=\delta \sum_{K\in\T} \sum_{\sigma \in \E_K} \tau_\sigma \Dd_{K,\sigma} (w^n-\langle w^n{\rangle}) \, { z^n_K} + \sum_{K \in \T} \sum_{\sigma \in \E_K} \tau_\sigma \Dd_{K,\sigma}(u^n \gamma(v^n)-\langle u^n \gamma v^n\rangle) z^n_K.
\end{align*}
For $I_1$, since $\langle z^n \rangle =0$ we have
\begin{align*}
    I_1 &=\eps \sum_{K \in \T} \m(K) \frac{(w^n_K-\langle w^n \rangle)-(w^{n-1}_K-\langle w^{n-1}\rangle)}{\Delta t} \, z^n_K + \beta \sum_{K\in\T} \m(K) (w^n_K-\langle w^n\rangle) \, z^n_K\\
    &=- \frac{\eps}{\Delta t}\sum_{K\in\T} \left(\sum_{\sigma \in \E_K} \tau_\sigma \left( \Dd_{K,\sigma} z^n - \Dd_{K,\sigma} z^{n-1} \right) \right) z^n_K - \beta \sum_{K\in\T} \left(\sum_{\sigma \in \E_K} \tau_\sigma \Dd_{K,\sigma} z^n \right) z^n_K.
\end{align*}
Therefore, thanks to a summation by parts we have
\begin{align*}
    I_1 = \frac{\eps}{\Delta t} \sum_{\substack{\sigma \in \E_{{\rm int}}\\\sigma=K|L}} \tau_{\sigma} \left(\Dd_{K,\sigma} z^n - \Dd_{K,\sigma} z^{n-1} \right) \Dd_{K,\sigma} z^n + \beta \sum_{\substack{\sigma \in \E_{{\rm int}}\\\sigma=K|L}} \tau_\sigma \left(\Dd_\sigma z^n \right)^2.
\end{align*}
It remains to apply the elementary inequality $2a(a-b) \geq a^2-b^2$ for every $a$ and $b \in \R$ and we get
\begin{equation}\label{6.I1}
    I_1 \geq \frac{\eps}{2} \sum_{\sigma \in \E} \tau_\sigma \frac{\left(\Dd_\sigma z^n\right)^2-\left(\Dd_\sigma z^{n-1}\right)^2}{\Delta t} + \beta \sum_{\sigma \in \E} \tau_\sigma \left(\Dd_\sigma z^n \right)^2.
\end{equation}
Now, for $I_2$ we use two discrete integration by parts to obtain
\begin{align*}
    I_2 &= \delta \sum_{K\in\T} \left(w^n_K - \langle w^n \rangle \right) \sum_{\sigma \in \E_K} \tau_\sigma \Dd_{K,\sigma} z^n + \sum_{K\in\T} \left(u^n_K \gamma(v^n_K)-\langle u^n \gamma(v^n)\rangle\right) \sum_{\sigma \in \E_K} \tau_\sigma \Dd_{K,\sigma} z^n\\
    &=-\delta \sum_{K\in \T} \m(K) \, \left(w^n_K-\langle w^n\rangle\right)^2 - \sum_{K\in\T} \m(K)  \left(u^n_K \gamma(v^n_K)-\langle u^n \gamma(v^n)\rangle\right) \, (w^n_K - \langle w^n \rangle ).
\end{align*}
Hence, thanks to Young's inequality we end up with
\begin{align}\label{6.I2}
    I_2 \leq -\frac{\delta}{2} \sum_{K\in \T} \m(K) \, \left(w^n_K-\langle w^n\rangle\right)^2 +\frac{1}{2\delta} \sum_{K\in\T} \m(K)  \left(u^n_K \gamma(v^n_K)-\langle u^n \gamma(v^n)\rangle\right)^2.
\end{align}
Thus, combining~\eqref{6.I1} and~\eqref{6.I2} yields~\eqref{6.unifdtv}.
\end{proof}

\begin{proposition}\label{prop.massw}
Assume \emph{\textbf{(H1)}--\textbf{(H4)}}. Then, the following estimates hold:
when $\eps>0$,
\begin{multline}\label{6.massw}
    \sum_{n=0}^{N_T-1} \Delta t \sum_{K\in\T} \m(K) \, |w^n_K|^2 \leq \dfrac{\m(\Omega)}{\beta} \,\left( \dfrac{\langle u^0 \rangle - \beta \langle v^0\rangle}{\sqrt{\eps}}\right)^2 \, \left(1-\left(\dfrac{1}{1+\frac{\beta \Delta t}{\eps}}\right)^{2N_T} \right)\\
    +\frac{2}{\delta^2} \sum_{n=0}^{N_T} \Delta t \| u^n \gamma(v^n) - \langle u^n \gamma(v^n) \rangle \|^2_{L^2(\Omega)},
\end{multline}
and when $\eps=0$,
\begin{equation}\label{6.masswQS}
    \sum_{n=0}^{N_T-1} \Delta t \sum_{K\in\T} \m(K) \, |w^n_K|^2
    \leq \frac{2}{\delta^2} \sum_{n=0}^{N_T} \Delta t \| u^n \gamma(v^n) - \langle u^n \gamma(v^n) \rangle \|^2_{L^2(\Omega)}.
\end{equation}
\end{proposition}

\begin{proof}
Thanks to the estimate~\eqref{6.unifdtv} and Young's inequality we have for each $0\leq n \leq N_T-1$
\begin{align*}
    \frac{\delta}{4}\sum_{K\in\T} \m(K) \, |w^n_K|^2 \leq \frac{\delta}{2} \, \m(\Omega) \, \langle w^n \rangle^2 + \frac{1}{2\delta} \sum_{K\in\T} \m(K) \, \left(u^n_K \gamma(v^n_K) - \langle u^n \gamma(v^n) \rangle \right)^2,
\end{align*}
so that
\begin{equation}\begin{split}\label{6.massw1}
    \sum_{n=0}^{N_T-1} \Delta t &\sum_{K\in \T  } \m(K) \, |w^n_K|^2 \\
    &\leq 2 \m(\Omega) \sum_{n=0}^{N_T-1} \Delta t \, \langle w^n \rangle^2 + \frac{2}{\delta^2} \sum_{n=0}^{N_T} \Delta t \| u^n \gamma(v^n) - \langle u^n \gamma(v^n) \rangle \|^2_{L^2(\Omega)}.
\end{split}\end{equation}
Let us now estimate the first term in the right hand side of~\eqref{6.massw1}. For this purpose, we use~\eqref{6.schw} and the mass conservation~\eqref{2.massu} to compute for any $1\leq n \leq N_T-1$,
\begin{align*}
    (\eps + \beta  \Delta t) \, \langle w^n \rangle = \eps\langle w^{n-1} \rangle.
\end{align*}
When $\eps=0$ we get that $\langle w^n \rangle=0$ for any $1\leq n \leq N_T-1$, so that~\eqref{6.masswQS} is a straightforward consequence of~\eqref{6.massw1}.
Now when $\eps>0$, a direct induction leads to, for $1\leq n \leq N_T-1$,
\begin{align*}
    \langle w^n \rangle = \dfrac{\langle w^{0} \rangle}{(1+\xi)^n},
\end{align*}
where $\xi :=\frac{\beta  \Delta t}\eps$. Besides, by definition of $w$ and using~\eqref{2.massv} at $n=1$ we have
\begin{align*}
    \langle w^0 \rangle = \frac1{\Delta t} \left( \frac{\langle u^0 \rangle}{\beta}   + \dfrac{1}{1+\xi} \left(\langle v^0\rangle -\frac{\langle u^0 \rangle}{\beta} \right) - \langle v^0 \rangle \right)  = \dfrac{\langle u^0 \rangle - \beta \langle v^0 \rangle}{\eps} \, \dfrac{1}{1+\xi}.
\end{align*}
Thus, we can compute
\begin{align*}
    \sum_{n=0}^{N_T-1} \Delta t \, \langle w^n \rangle^2 &= \dfrac{(\langle u^0 \rangle - \beta \langle v^0\rangle)^2}{\eps^2} \, \Delta t \sum_{n=0}^{N_T-1} \dfrac{1}{\left(1+\xi\right)^{2n+2}},
\end{align*}
where
\begin{align*}
\Delta t \sum_{n=0}^{N_T-1} \dfrac{1}{\left(1+\xi\right)^{2n+2}}
= \dfrac{\Delta t}{\left(1+\xi\right)^2} \, \dfrac{1-\frac{1}{(1+\xi)^{2N_T}} }{1-\frac{1}{(1+\xi)^2}}
= \dfrac{\Delta t}{2\xi} \, \dfrac{1-\frac{1}{(1+\xi)^{2N_T}}}{1+\frac{\xi}{2}}
= \dfrac{\eps}{2\beta} \, \dfrac{1-\frac{1}{(1+\xi)^{2N_T}}}{1+\frac{\xi}{2}}.
\end{align*}
Hence, it holds
\begin{align}\label{6.massw2}
    2 \m(\Omega) \sum_{n=0}^{N_T-1} \Delta t \, \langle w^n \rangle^2 \leq \dfrac{\m(\Omega)}{\beta} \,\left( \dfrac{\langle u^0 \rangle - \beta \langle v^0\rangle}{\sqrt{\eps}}\right)^2 \, \left(1-\left(\dfrac{1}{1+\frac{\beta \Delta t}{\eps}}\right)^{2N_T} \right).
\end{align}
Now, we combine the estimates~\eqref{6.massw1} and~\eqref{6.massw2} and we obtain~\eqref{6.massw}.
\end{proof}

\begin{corollary}\label{coro.estimdtv_unifeps}
Assume \emph{\textbf{(H1)}--\textbf{(H5)}} and let $\D$ be an admissible mesh satisfying~\eqref{2.dd}. Then, there exists a constant $C>0$ only depending on $\delta$, $\beta$, $T$, $\Omega$ and $\zeta$ such that
\begin{align}
    \sum_{n=0}^{N_T-1} \Delta t \sum_{K\in\T} \m(K) \, \left|\dfrac{v^{n+1}_K-v^n_K}{\Delta t}\right|^2 \leq  C \, \left( O(1) + 1\right).
\end{align}
\end{corollary}

\begin{proof}
The result is a direct consequence of Proposition~\ref{prop.massw}, the definition \eqref{6.defw}, the duality estimate~\eqref{4.estdual} and assumption \textbf{(H5)}.
\end{proof}

We can now make use of these uniform estimates to show the AP property of the scheme.

\begin{proof}[Proof of Theorem~\ref{thm.AP}]
The result is a direct consequence of the previous uniform estimates and a slight adaptation of the computations done in Section~\ref{sec.limitscheme} in order to identify the limit functions as weak solutions to the quasi-stationary system in the sense of Definition~\ref{def.weak}.
\end{proof}

\section{Numerical experiments}\label{sec.numexp}

\subsection{Implementation}\label{sec:imp}

The scheme was implemented in dimension $d=1$ and $d=2$ using Matlab. The code is available at {\url{https://gitlab.inria.fr/herda/fvlocsens}}. For bi-dimensional simulations, we use unstructured triangular meshes generated by Gmsh with a Delaunay algorithm, interfaced with Matlab through the code available at \url{https://gitlab.inria.fr/herda/fvmeshes}.  As the scheme is linearly implicit it involves one resolution of linear system per species per time step. The matrix for the equation on $u$ depends on the second species and must be reassembled at each time step. In order to optimize the computational cost, the mass and the stiffness matrices should be preassembled outside of the time loop.

\subsection{Testcase 1. Convergence and entropy decay}
In this first testcase we illustrate and complement our theoretical convergence result with an experimental convergence test in dimension $d=1$. The domain is taken to be the segment $\Omega = [0,1]$, on which we consider the system~\eqref{1.equ}-\eqref{1.IC} endowed with the parameters 
\[
\varepsilon = 10^{-3},\quad \delta = 10^{-3},\quad \beta = 0.1,
\]
and the initial conditions 
\[
u^0(x) = 15x^2(1-x)^2,\quad  v^0(x) = 0.
\]
{We solve the scheme~\eqref{2.schu}-\eqref{2.schv} to compute an approximate solution of the initial system at final time $T = 20$. We run the scheme for a sequence of decreasing space steps. More precisely the number of points is $N_k = 50\cdot2^{k}$ for $k=0,\dots,6$. The time step $\Delta t = 10^{-4}$ has been experimentally chosen small enough so that spatial errors dominate. As we do not know the analytical solution for this system, we take as reference solution the computed solution on the finest mesh ($3200$ cells). Then the error for the $k$-th mesh is taken to be the $L^2$ or $L^\infty$ distance between the $k$-th solution and the reference solution projected on the $k$-th mesh. The results are reported in Table~\ref{tab:convergence}.}
\begin{table}[!ht]

\begin{tabular}{c|c||c|c||c|c}
$k$&$N_k$ &$L^2$ error      &order&$L^\infty$ error & order\\\hline
$0$&$50$  &$6.4\cdot10^{-2}$&     &$1.0\cdot10^{-1}$&\\
$1$&$100$ &$1.5\cdot10^{-2}$&$2.1$&$2.3\cdot10^{-2}$&$2.2$\\
$2$&$200$ &$3.7\cdot10^{-3}$&$2.0$&$6.7\cdot10^{-3}$&$1.8$\\
$3$&$400$ &$9.1\cdot10^{-4}$&$2.0$&$1.7\cdot10^{-3}$&$2.0$\\
$4$&$800$ &$2.2\cdot10^{-4}$&$2.1$&$4.0\cdot10^{-4}$&$2.1$
\end{tabular}
\caption{Testcase 1. Estimated errors and orders of convergence in space at final time $T =20$. Reference solution is for $3200$ cells and $\Delta t = 10^{-4}$.}
\label{tab:convergence}
\end{table}
Convergence in the space variable is experimentally observed. This illustrates the result of Theorem~\ref{thm.conv} and suggests, as expected {(see~\cite{DroNat18_super_conv})}, a second order accuracy in space. {In fact, while Theorem~\ref{thm.conv} predicts a weak convergence in $L^1$ (for a subsequence), the experimental observations suggest a strong convergence in stronger norms. Here we present only one experiment, but these results were tested on multiple configurations and appear to be robust.}

On the same testcase, for the finest mesh, we report on Figure~\ref{fig:entropy} the time evolution of the entropy and each term in the sum composing it. This illustrates in particular the decay of the entropy as predicted in Theorem~\ref{thm.exi}. It also highlights that the decay of the entropy does not imply a decay for the sum of the Boltzmann entropy for $u$ and the $H^1$ norm of $v$, but is instead associated with a large compensation between that sum and the negative term $-\int_\Omega uv \, \mathrm{d}x$. Finally it illustrates that the problem of establishing a lower bound for the entropy is far from artificial.
\begin{figure}[!ht]
\centering
\includegraphics[width = .6\textwidth]{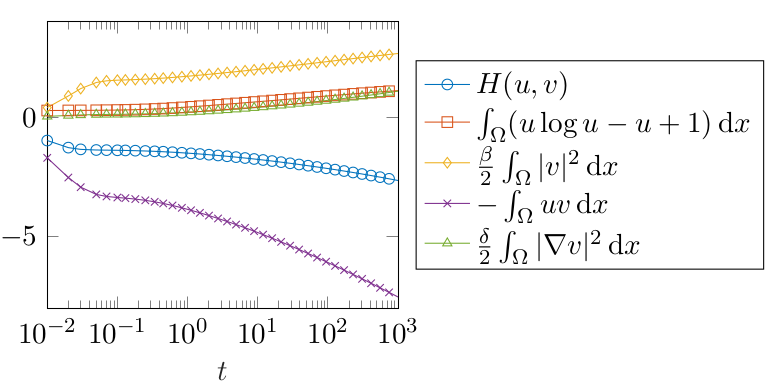}
\caption{Testcase 1. Time evolution of the entropy and related quantities. The mesh is taken as $N=3200$ and $\Delta t = 10^{-2}$.}\label{fig:entropy}
\end{figure}

\subsection{Testcase 2. AP property in the quasi-stationary limit}\label{sec:num_AP}
In this testcase, we investigate the limit $\eps\to 0$. We start again from the domain 
$\Omega = [0,1]$ and the parameters are taken as $\delta=10^{-3}$ and $\beta=0.1$. We consider three different  initial conditions which are more or less well-prepared with respect to the limit equation. The first initial condition, which we refer to as the \emph{strongly well-prepared} (SWP) initial condition, is given by
\begin{equation}\label{eq:strong_well_prepared_init}
u^0(x) = 15x^2(1-x)^2,\quad  -\delta\Delta v^0(x) + \beta v^0(x) = u^0, \quad \beta\langle v^0\rangle = \langle u^0\rangle.
\end{equation}
It satisfies assumption \textbf{(H5)} and moreover $v^0$ and $u^0$ satisfy the limit Poisson equation. The second initial condition, which we shall refer to as the \emph{well-prepared} (WP) initial condition is 
\begin{equation}\label{eq:well_prepared_init}
u^0(x) = 15x^2(1-x)^2,\quad  \beta v^0(x) = \langle u^0\rangle.
\end{equation}
Once again assumption \textbf{(H5)} is satisfied, however $-\delta\Delta v^0(x) + \beta v^0(x) \neq u^0$. The third initial condition is referred to as the \emph{ill-prepared} (IP) initial condition and coincides with that of the previous convergence test case, namely
\begin{equation}\label{eq:ill_prepared_init}
u^0(x) = 15x^2(1-x)^2,\quad  v^0(x) = 0.
\end{equation}
In this case assumption \textbf{(H5)} is not satisfied. We simulate the equation on a uniform one dimensional mesh composed of $3200$ cells with final time $T=100$. As we want to emphasize the initial layer at $t=0$ which may arise depending on the compatibility of $u^0$ and $v^0$ with the limit equation, the time step is taken of the order of $10^{-8}$ until time $10^{-2}$. Then, on the time interval $[10^{-2}, T]$ the time step is taken to be $\Delta t = 10^{-2}$.

 On Figure~\ref{fig:APprofiles}, we show the evolution of the approximate solution at different times, in the cases $\eps = 10^{-1}$ and $\eps = 0$, for the ill-prepared initial condition \eqref{eq:ill_prepared_init}. 
\begin{figure}[!ht]
\centering
\includegraphics[width = .8\textwidth]{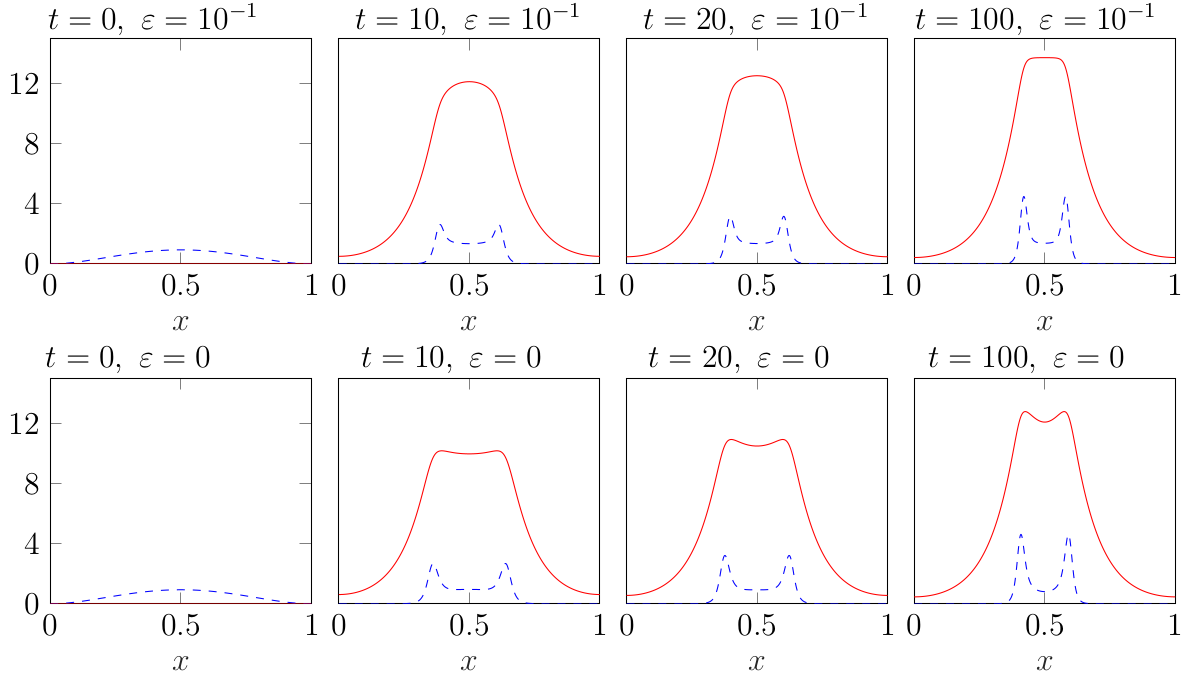}
\caption{Testcase 2. Snapshots of the chemoattractant $v$ (red), and the density $u$ (dashed blue), at time $t\in\{0, 10, 20, 100\}$ for $\eps = 10^{-1}$ (top line) and $\eps = 0$ (bottom line). The initial condition is $u^0(x) = 15x^2(1-x)^2$,  $v^0(x) = 0$.}\label{fig:APprofiles}
\end{figure}

On Figure~\ref{fig:APdtV}, we show $\|\langle\partial_tv\rangle\|_{L^2(0,T)}$ with respect to $\eps$ for the different initial data. This is an illustration of \eqref{eq:evoldtV} and its discrete counterpart \eqref{6.massw2}. Whenever the well-preparedness assumption \textbf{(H5)} is not satisfied then $\|\langle\partial_tv\rangle\|_{L^2(0,T)}$ tends to $+\infty$ as $\eps\to0$.
\begin{figure}[!ht]
\centering
\includegraphics[width = .6\textwidth]{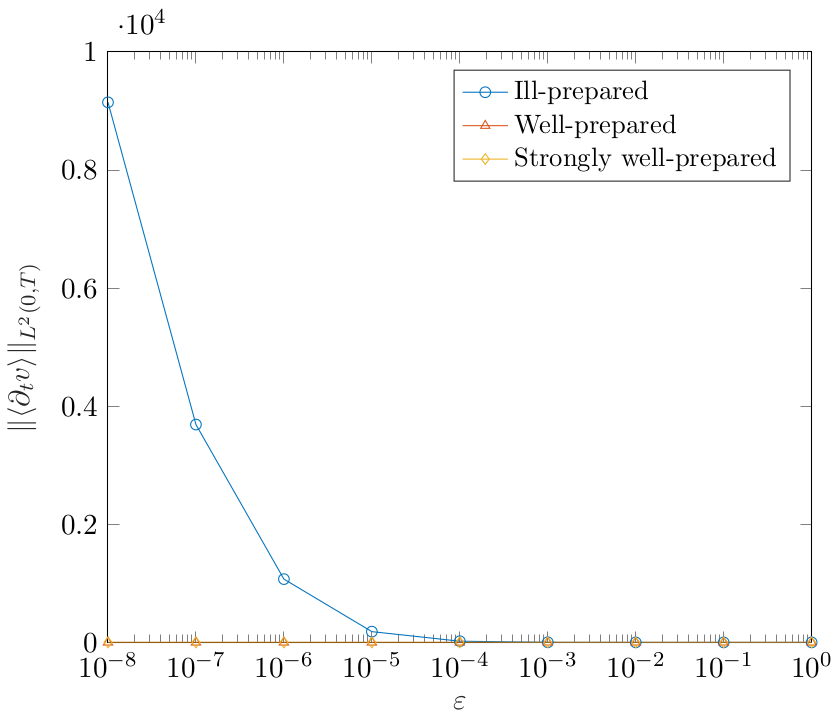}
\caption{Testcase 2. Time derivative of the chemoattractant concentration in the quasi-stationary limit for ill-prepared  \eqref{eq:ill_prepared_init}, well-prepared \eqref{eq:well_prepared_init}, and strongly well-prepared \eqref{eq:strong_well_prepared_init} initial conditions.}\label{fig:APdtV}
\end{figure}

On Figure~\ref{fig:APerrors} we plot the evolution of the error between the solution to the scheme and the solution of the limit ($\eps=0$) scheme as $\eps\to0$ with respect to time for different values of $\eps$ and different initial data. Even if this does not per se illustrate the conclusion Theorem~\ref{thm.AP} (as the time and step size are fixed here), the interesting point is that convergence in $L^2(Q_T)$ norm is satisfied under \textbf{(H5)}. However it is not necessarily satisfied in $L^\infty(0,T; L^2(\Omega))$ norm under the same hypothesis. It holds only for the strongly well-posed initial data \eqref{eq:strong_well_prepared_init} where the limit equation is enforced on the initial data which prevents an initial boundary layer. A more surprising observation is that the $L^2(Q_T)$ convergence is also satisfied for the ill-prepared initial condition. This suggests that our assumption \textbf{(H5)} might be weakened.
\begin{figure}[!ht]
\centering
\includegraphics[width = .8\textwidth]{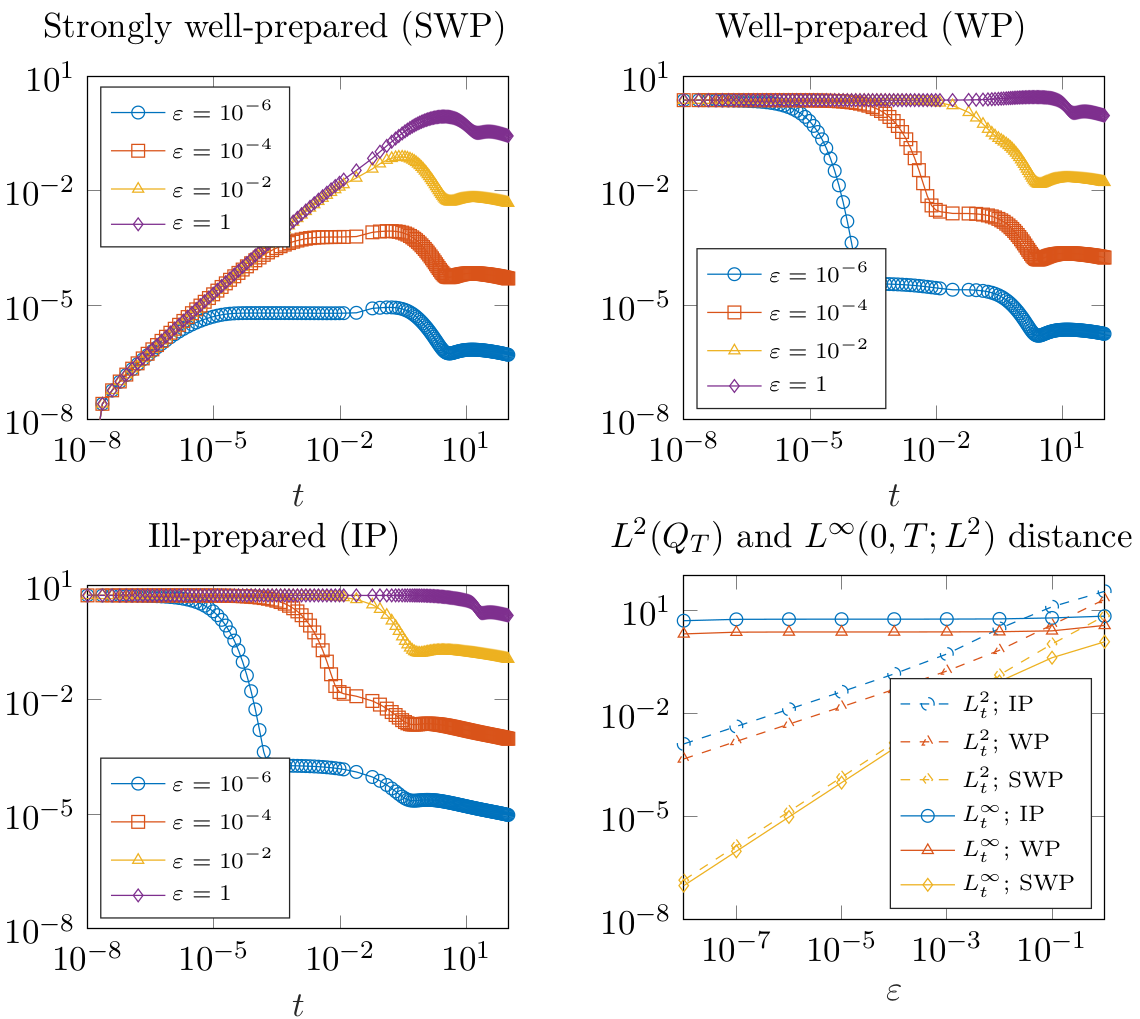}
\caption{Testcase 2. Time evolution of the error (in $\log$-$\log$ scale) between the chemoattractant density $v$ for a given $\eps$ and the quasi-stationary case $\eps=0$. (Top left) $L^2(\Omega)$ distance vs time for the strongly well-prepared (SWP) initial data \eqref{eq:strong_well_prepared_init}, (top right) $L^2(\Omega)$ distance vs time for the well-prepared (WP) initial data \eqref{eq:well_prepared_init} and (bottom left) $L^2(\Omega)$ distance vs time for the ill-prepared (IP) initial data \eqref{eq:ill_prepared_init}. On the bottom right plot we show the corresponding $L^2(Q_T)$ and $L^\infty(0,T; L^2(\Omega))$ errors with respect to $\eps$ for the three initial data.}\label{fig:APerrors}
\end{figure}

\subsection{Testcase 3. Instabilities in 2D}

In this testcase, we investigate numerically the linear stability of homogeneous steady states  in a bi-dimensional domain. Given $\mu>0$, consider  the homogeneous steady state 
\begin{equation}\label{eq:hom_steady}
(u_*,v_*) = (\mu, \mu/\beta).
\end{equation}
A straightforward adaptation of the linear stability analysis performed in 
\cite[proof of Theorem 3.1]{DKTY2019} and \cite[Theorem 3.2]{ChKi24} yields that \eqref{eq:hom_steady} is asymptotically linearly stable if and only if 
\begin{equation}\label{eq:lin_stab}
\mu<\mu_c^s \quad\text{with}\quad \mu_c^s = \beta + \lambda_1(\Omega)\delta,
\end{equation} 
where $\lambda_1(\Omega)>0$ is the first non-zero eigenvalue of the Laplacian $(-\Delta)$ in $\Omega$ with homogeneous Neumann boundary conditions.  This linear stability analysis can be transposed in the fully discrete setting. This yields, independently of $\Delta t>0$, the same stability condition as \eqref{eq:lin_stab} where $\lambda_1(\Omega)$ is replaced with the first non-zero eigenvalue of the finite volume Laplacian $\mathbb{L}$ associated with the mesh, defined by
\begin{equation}\label{eq:FV_Lap}
\left\{
\begin{array}{ll}
\mathbb{L}_{K,K} =\sum_{\sigma \in \E_{{\rm int},K}} \dfrac{\tau_\sigma}{\m(K)},  &\forall K \in \T,\\
\mathbb{L}_{K,L} = -\dfrac{\tau_\sigma}{\m(K)},  &\forall K \in \T, \forall L \in \Ne(K), \, \sigma=K|L,	 \\
\mathbb{L}_{K,L} = 0,  &\forall K \in \T, \forall L \notin \Ne(K).
\end{array}
\right.
\end{equation}

The space domain $\Omega$ is taken as a (polygonal approximation of) a circle of radius $R = 1$ centered at the origin. In a disk domain of radius $R$ the principal eigenvalue of the Laplacian is $\lambda_1(\Omega) = \left(j_{1,1}'/R\right)^2$ with $j_{1,1}'\approx 1.8412$ the first non-zero root of $J_1'$, with $J_1$ the first Bessel function. The parameters are taken to be $R=\delta=\eps=1$ and $\beta = \lambda_1(\Omega) / R^2\approx3.39$. This yields $\mu_c^s = 2\beta \approx6.78$.

The initial data is taken to be a perturbation of $(u_*,v_*)$. In polar coordinates with origin  at the center of the disk, we set either a perturbation with the first non-trivial mode of the Neumann Laplacian
\begin{equation}\label{eq:initJ1}
u^0(r,\theta) = \mu,\quad  v^0(r,\theta) = \mu\left(1 + \frac{\cos(\theta)}{10}J_1\left(\sqrt{\lambda_1(\Omega)}\frac{r}{R}\right)\right),
\end{equation}
or the third non-trivial mode of the Neumann Laplacian (which is a radially symmetric function)
\begin{equation}\label{eq:initJ3}
u^0(r,\theta) = \mu,\quad  v^0(r,\theta) = \mu\left(1 + \frac{1}{10}J_0\left(\sqrt{\lambda_3(\Omega)}\frac{r}{R}\right)\right).
\end{equation}
We test several values of $\mu$ that are 
\[
\text{(a)}\ \mu = 0.9\mu_c^s,\qquad \text{(b)}\ \mu = 1.1\mu_c^s,\qquad \text{(c)}\ \mu = 4\mu_c^s.
\]
For the initial data \eqref{eq:initJ1}, linear instability is expected in cases \text{(b)} and \text{(c)} since $\mu>\mu_c^s$. For the second initial data \eqref{eq:initJ3}, linear instability of the corresponding third eigenmode of the Neumann Laplacian is expected only in case \text{(c)} where one has  $4\mu_c^s>\beta + \lambda_3(\Omega)\delta$ with our parameters. Here $\lambda_3(\Omega)= \left(j_{0,1}'/R\right)^2$ 
with  $j_{0,1}'\approx 3.8317$ the first non-zero root of $J_0'$.

The domain is meshed with $11790$ triangles, the time step is taken to be $\Delta t = 0.1$ and the final time of simulation is $T=50000$. Let us precise that the mesh does not preserve any sort of radial symmetry of the domain.

On Figure~\ref{fig:snapshots_J1c} and Figure~\ref{fig:snapshots_J3c}, we report snapshots of the solution at different times for the initial data \eqref{eq:initJ1} and \eqref{eq:initJ3} in the case (c) $\mu = 4\mu_c^s$, so that both cases are linearly unstable with a different perturbation. As expected, the solution deviates largely from the homogeneous equilibrium and tends to localize in parts of the domain.

\begin{figure}[!ht]
	\centering
	\includegraphics[width = .73\textwidth]{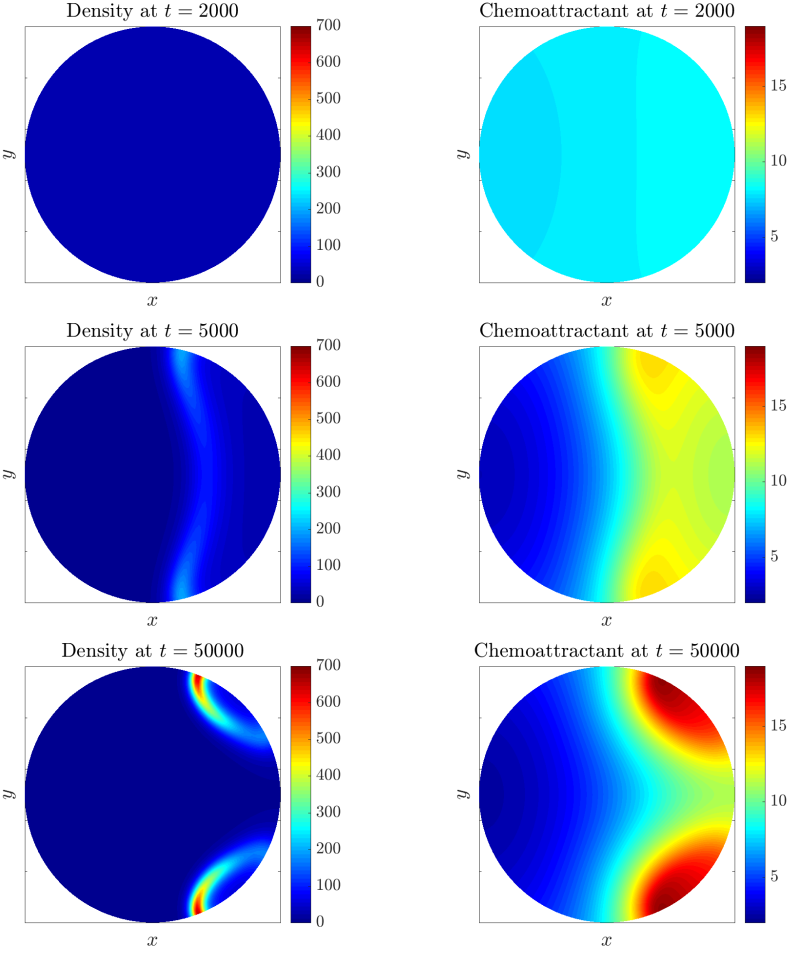}
    \caption{Testcase 3. Snapshots of the solution density $u$ (left) and chemoattractant $v$ (right) at different times $t=2000, 5000, 50000$ (top to bottom). Initial data is \eqref{eq:initJ1} with (c) $\mu = 4\mu_c^s$.}
    \label{fig:snapshots_J1c}
\end{figure}

\begin{figure}[!ht]	
	\centering
\includegraphics[width = .73\textwidth]{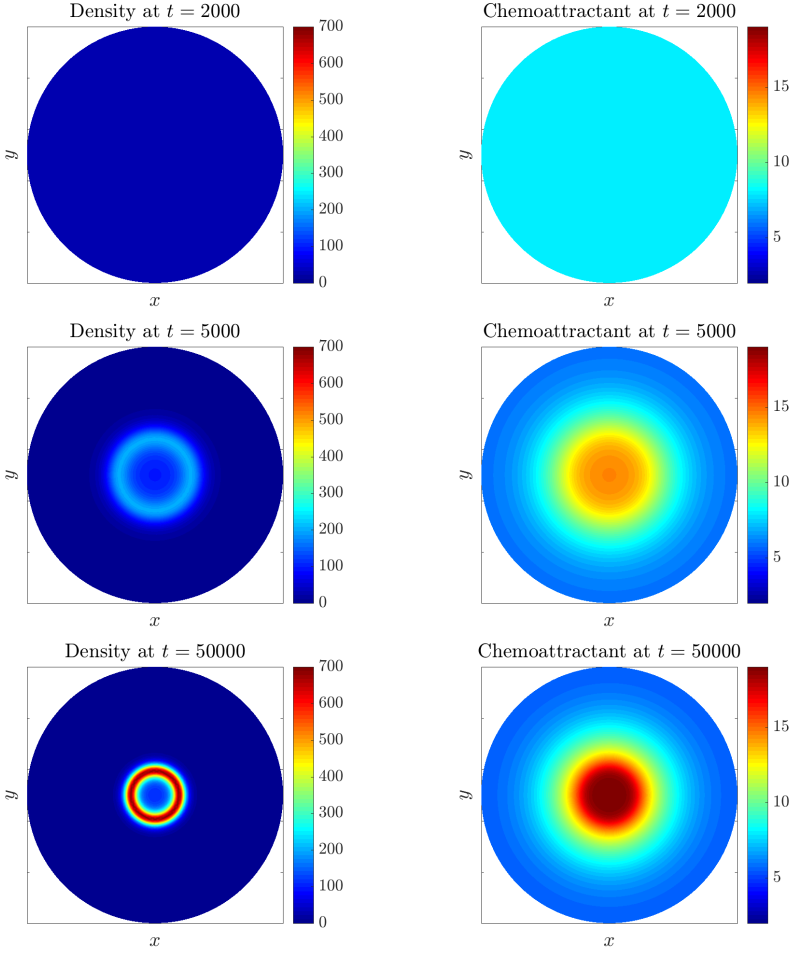}
    \caption{Testcase 3. Snapshots of the solution density $u$ (left) and chemoattractant $v$ (right) at different times $t=2000, 5000, 50000$ (top to bottom). Initial data is \eqref{eq:initJ3} with (c) $\mu = 4\mu_c^s$. }
    \label{fig:snapshots_J3c}
\end{figure}

On Figure~\ref{fig:Linf_and_entropy} we show the evolution of global quantities. In the stable case \eqref{eq:initJ1} (a), we plot the time evolution of the (discrete counterpart of the) relative entropy with respect to the homogeneous steady state
\[
\mathcal{H}((u,v); (u_*,v_*)) = \int_\Omega \left[(u-u_*) \log \left(\frac{u}{u_*}\right) + \frac{\beta}{2} |v-v_*|^2 -(u-u_*)(v-v_*) + \frac{\delta}{2} \left|\nabla v \right|^2 \right] \, \dist x.
\]
We observe exponential decay of this quantity which illustrates the exponential asymptotic stability of the homogeneous steady state for this set of parameters. In the unstable cases, namely \eqref{eq:initJ1} (b), \eqref{eq:initJ1} (c) and \eqref{eq:initJ3} (c), we show the time evolution of the $L^\infty$ norms with respect to time. The $L^\infty$ norm of $u$ is monotonically increasing for all three simulations.

In fact, we recall that the literature \cite{FuJi20,BLT21,FuJi21a,JiWa20} predicts a dichotomy between uniformly bounded solutions and {possible} infinite-time blow-up. {However, from our observations it is challenging in practice to judge whether a blow-up phenomenon is actually detected for a given set of parameters. Indeed, in regimes where the homogeneous steady state is unstable, both below and above the critical mass threshold, we observe aggregation patterns that appear qualitatively similar. At the discrete level, we emphasize that the growth of $\|u\|_{L^\infty}$ is constrained by positivity and mass conservation, and on coarse meshes this quantity may saturate. Moreover, even in parameter regimes where blow-up is theoretically expected, the growth of $\|u\|_{L^\infty}$ is extremely slow, in line with the infinite-time blow-up scenario, and hence difficult to distinguish from a merely bounded but increasing solution. A thorough numerical investigation of this phenomenon, including refined meshes and adaptive strategies, will be the subject of forthcoming work.}

\begin{figure}[!ht]
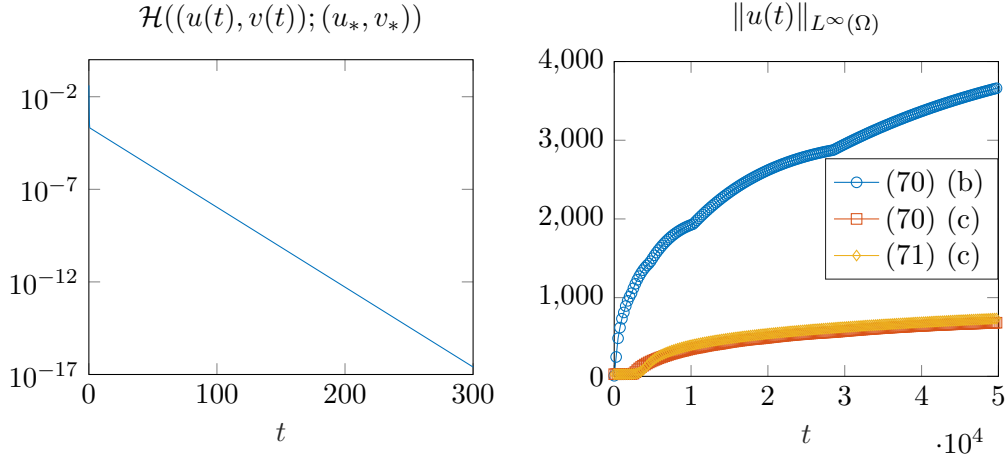

\begin{tabular}{cc}
\begin{minipage}{.43\textwidth}  \input{Entropy_L2_dt0.1_delt1_eps1_mu6.102_mesh6.tex}\end{minipage}
&
\begin{minipage}{.43\textwidth} 
\input{NormsU_L2_dt0.1_delt1_eps1_mu27.1201_mesh6}
\end{minipage}
\end{tabular}
	\caption{Testcase 3.(Left) In the stable case \eqref{eq:initJ1} (a): time evolution of the relative entropy between the solution and the homogeneous steady state. (Right) In the unstable cases \eqref{eq:initJ1} (b), \eqref{eq:initJ1} (c) and \eqref{eq:initJ3} (c): $L^\infty$ norm of $u$  versus time.}
    \label{fig:Linf_and_entropy}
\end{figure}

\subsection{Testcase 4. Aggregation in large domain}\label{sec:DKTY_testcase}

We reproduce here the testcase of \cite[Section 4.2.2]{DKTY2019} in order to illustrate the use of the scheme for a non-exponential cell motility. The latter is taken to be 
\[
\gamma(v) = \frac{1}{c+v^k}\,,
\]
with $c=1$ and $k=2$. The space domain $\Omega$ is taken as a square with edge of length $R = 10$. The parameters are taken to be $\eps=\beta =1$ and $\delta = 3/(5\times14.6819)$ The initial data is taken to be a perturbation of the homogeneous steady state $(u_*,v_*)$ defined in \eqref{eq:hom_steady} with $\mu=4$
\[
u^0(x,y) = \mu,\quad  v^0(x,y) = \mu\left(0.5 + X(x,y)\right),
\]
where $0<X(x,y)<1$ is the realisation of a random variable with uniform distribution. The value of the parameter $\delta$ has been chosen in order for the homogeneous steady state to be linearly unstable as in the previous section. The domain is meshed with $14788$ triangles, the time step is taken to be $\Delta t = 0.1$ and the final time of simulation is $T=20000$. Let us point out that compared to  \cite[Section 4.2.2]{DKTY2019}, we can handle a non-uniform mesh and a coarser time step. For the latter point, we recall that our semi-implicit scheme is unconditionally stable, unlike the explicit scheme used in \cite[Section 4.2.2]{DKTY2019}. On Figure~\ref{fig:snapshots_DKTY2019}, we show the evolution of the density with respect to time. As pointed out in \cite[Section 4.2.2]{DKTY2019}, we observe aggregation  in spots which tend to merge in time. 

\begin{figure}[!ht]
	\centering
	\includegraphics[width = .8\textwidth]{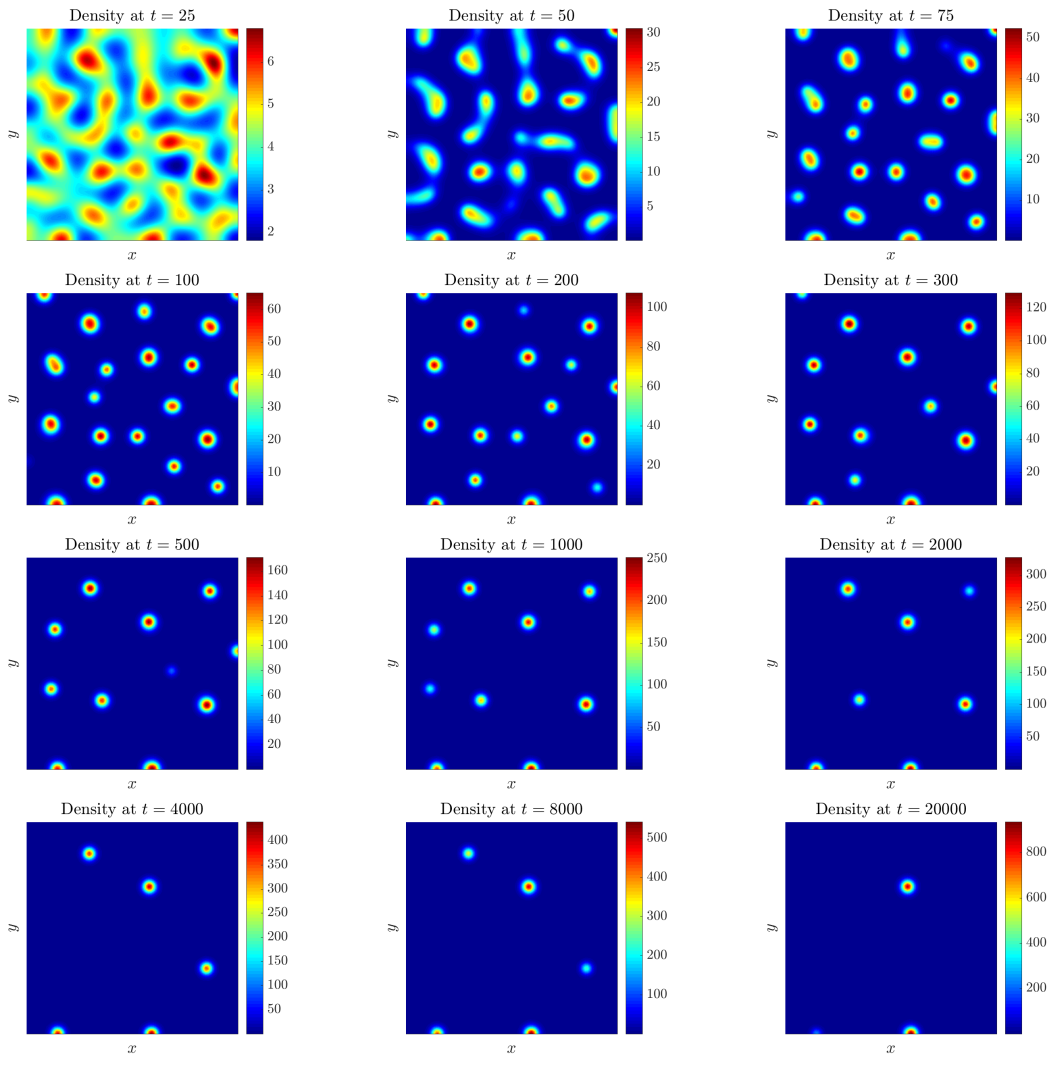}
    \caption{Testcase 4. Snapshots of the density at times (left to right, top to bottom)  $t=25$, $50$, $75$, $100$, $200$, $300$, $500$, $1000$, $2000$, $4000$, $8000$, $20000$. The color scale varies between each subfigure.}
    \label{fig:snapshots_DKTY2019}
\end{figure}

\begin{appendix}

\section{Auxiliary results}\label{app.A}

\begin{lemma}\label{lem.con}
Let $\T$ be an admissible mesh of $\Omega$ satisfying~\eqref{2.dd} and $w \in H^1(\Omega)$. Then,  the vector $(w_K)_{K \in \T} \in \Hi_\T$ with $w_K = \m(K)^{-1}\int_K w(x) \, \dist x$ satisfies
\begin{align*}
\sum_{K \in \T} \m(K) |w_K|^2 \leq \int_\Omega |w(x)|^2 \, \dist x.
\end{align*}
\end{lemma}

\begin{proof}
Thanks to Jensen's inequality we have
\begin{align*}
\sum_{K \in \T} \m(K) |w_K|^2 - \int_\Omega |w(x)|^2 \, \dist x &= \sum_{K\in\T} \left(\frac{1}{\m(K)} \left(\int_K w(x)\,\dist x\right)^2 - \int_K |w(x)|^2 \, \dist x \right)\\
&= \sum_{K\in\T}\left( \m(K) \left(\int_K w(x)\,\frac{\dist x}{\m(K)}\right)^2 - \int_K |w(x)|^2 \, \dist x \right)\\
&\leq \sum_{K\in\T}\left( \m(K) \int_K |w(x)|^2\,\frac{\dist x}{\m(K)} - \int_K |w(x)|^2 \, \dist x \right) = 0.
\end{align*}
This finishes the proof of Lemma~\ref{lem.con}.
\end{proof}

\begin{lemma}[see Lemma 3.4 in~\cite{EGH00}]\label{lem.H1con}
Let $\T$ be an admissible mesh of $\Omega$ satisfying~\eqref{2.dd} and $w \in H^1(\Omega)$. Then, there exists a constant $C>0$ only depending on $\Omega$ and $\zeta$ such that the vector $(w_K)_{K \in \T} \in \Hi_\T$ with $w_K = \m(K)^{-1}\int_K w(x) \, \dist x$ satisfies
\begin{align*}
\sum_{\sigma \in \E} \tau_\sigma (\Dd_{\sigma} w)^2 \leq C \|w\|_{H^1(\Omega)}^2.
\end{align*}
\end{lemma}

\begin{lemma}[see Proposition 2.61 in~\cite{FoLe07}]\label{lem.convugamm}
Let $(X,\mathcal{M},\mu)$ be a measurable space with $\mu$ finite. If $u_n \rightharpoonup u$  weakly in $L^1(X)$, $v_n \to v$ pointwise for $\mu$ a.e. $x \in X$, and $\sup_n\|v_n\|_{L^\infty(X)}<\infty$, then $u_n v_n \rightharpoonup u v$ weakly in $L^1(X)$.
\end{lemma}

\section{On the continuity of the discrete projection} \label{sec.countex}

\begin{proposition}[Continuity of the projection, counter-example]
Let $\Omega=(-1,1)$ and, for $m\ge 2 $, let $\T$ be given by the regular grid, $\T=\left\{\left(\frac{k}{m},\frac{k+1}{m}\right), \,\, - m \le k \le m-1 \right\}$. Then,
\begin{equation}
    \sup_{u \in (H^1)'(\Omega)} \frac{N(\pi_\T u)}{\|u\|_{(H^1)'(\Omega)}} = \infty.
\end{equation}
\end{proposition}

\begin{proof}
Let $r\ge m$ and $u=\delta_{\frac1r}-\delta_{-\frac1r}$. Clearly $\langle u \rangle =0$ and for any $\varphi\in H^1(\Omega)$, we have
\begin{equation*}
    \langle u , \varphi \rangle = \varphi\left(\frac1r\right)-\varphi\left(-\frac1r\right)=\int_{-\frac1r}^\frac1r \varphi'\le \frac{\sqrt{2}}{\sqrt{r}} \|\varphi'\|_2,
\end{equation*}
so that $u\in (H^1)'(\Omega)$ with $\|u\|_{(H^1)'(\Omega)}\le \frac{\sqrt{2}}{\sqrt{r}}$.

On the other hand, when projecting $u$ on the grid $\T$, using that $r\ge m$ we get that $\pi_\T u = m \mathbf{1}_{(0,\frac1m)}-m \mathbf{1}_{(\frac{-1}m,0)}$. We then pick the test function $\theta = \frac{1}{2\sqrt{m}}\left( \mathbf{1}_{(0,1)}-\mathbf{1}_{(-1,0)}\right)$ and verify that $\theta \in \widetilde{\Hi}_T$ with $|\theta|_{1,2,\T} = 1$, so that, using Lemma~\ref{lem.dual}, we get
\begin{equation*}
    N(\pi_\T u) \ge \left|\int_{-1}^1 \pi_\T u \, \theta \, \dist x \right|= \frac{1}{\sqrt{m}}.
\end{equation*}
Taking $r=A m^2$ with $A>1$, we obtain
\begin{equation*}
    \frac{N(\pi_\T u)}{\|u\|_{(H^1)'}} \ge \sqrt{\frac{Am}{2}},
\end{equation*}
and we conclude by letting $A$ tend to infinity.
\end{proof}

\end{appendix}
\section*{Acknowledgements} M.H. acknowledges support from the LabEx CEMPI (ANR-11-LABX0007) and the CDP C2EMPI, together with the French State under the France-2030 programme, the University of Lille, the Initiative of Excellence of the University of Lille, the European Metropolis of Lille for their funding and support of the R-CDP-24-004-C2EMPI project.  A.Z. acknowledges support from the Research Department of UTC as part of the Young Lecturers Call for Proposals. The premises of this work were discussed during the workshop MoDiS 2023: Modelling Diffusive Systems at ICMS, Edinburgh.

\bibliographystyle{plain}
\bibliography{bibli}

\begin{thebibliography}{10}

\bibitem{ssAgPi16}
Oscar Agudelo and Angela Pistoia.
\newblock Boundary concentration phenomena for the higher-dimensional
  {Keller}-{Segel} system.
\newblock {\em Calc. Var. Partial Differ. Equ.}, 55(6):31, 2016.
\newblock Id/No 132.

\bibitem{AB17}
Mohammed Akhmouch and Mohammed Benzakour~Amine.
\newblock A time semi-exponentially fitted scheme for chemotaxis-growth models.
\newblock {\em Calcolo}, 54(2):609--641, 2017.

\bibitem{ABS11}
Boris Andreianov, Mostafa Bendahmane, and Mazen Saad.
\newblock Finite volume methods for degenerate chemotaxis model.
\newblock {\em J. Comput. Appl. Math.}, 235(14):4015--4031, 2011.

\bibitem{AT21}
Gurusamy Arumugam and Jagmohan Tyagi.
\newblock Keller-{Segel} chemotaxis models: a review.
\newblock {\em Acta Appl. Math.}, 171:82, 2021.
\newblock Id/No 6.

\bibitem{BCF15}
Marianne Bessemoulin-Chatard, Claire Chainais-Hillairet, and Francis Filbet.
\newblock On discrete functional inequalities for some finite volume schemes.
\newblock {\em IMA J. Numer. Anal.}, 35(3):1125--1149, 2015.

\bibitem{BCJ14}
Marianne Bessemoulin-Chatard and Ansgar J{\"u}ngel.
\newblock A finite volume scheme for a {Keller}-{Segel} model with additional
  cross-diffusion.
\newblock {\em IMA J. Numer. Anal.}, 34(1):96--122, 2014.

\bibitem{ssBil98}
Piotr Biler.
\newblock Local and global solvability of some parabolic systems modelling
  chemotaxis.
\newblock {\em Adv. Math. Sci. Appl.}, 8(2):715--743, 1998.

\bibitem{BCC08}
Adrien Blanchet, Vincent Calvez, and Jos{\'e}~A. Carrillo.
\newblock Convergence of the mass-transport steepest descent scheme for the
  subcritical {Patlak}-{Keller}-{Segel} model.
\newblock {\em SIAM J. Numer. Anal.}, 46(2):691--721, 2008.

\bibitem{ssBCF20}
Denis Bonheure, Jean-Baptiste Casteras, and Juraj Foldes.
\newblock Singular radial solutions for the {Keller}-{Segel} equation in high
  dimension.
\newblock {\em J. Math. Pures Appl. (9)}, 134:204--254, 2020.

\bibitem{ssBCN17}
Denis Bonheure, Jean-Baptiste Casteras, and Benedetta Noris.
\newblock Multiple positive solutions of the stationary {Keller}-{Segel}
  system.
\newblock {\em Calc. Var. Partial Differ. Equ.}, 56(3):35, 2017.
\newblock Id/No 74.

\bibitem{ssBCR21}
Denis Bonheure, Jean-Baptiste Casteras, and Carlos Rom{\'a}n.
\newblock Unbounded mass radial solutions for the {Keller}-{Segel} equation in
  the disk.
\newblock {\em Calc. Var. Partial Differ. Equ.}, 60(5):30, 2021.
\newblock Id/No 198.

\bibitem{BrPa24}
Maxime Breden and Maxime Payan.
\newblock Computer-assisted proofs for the many steady states of a chemotaxis
  model with local sensing.
\newblock {\em Physica D}, 466:18, 2024.
\newblock Id/No 134221.

\bibitem{BCR05}
C.~J. Budd, R.~Carretero-Gonz{\'a}lez, and R.~D. Russell.
\newblock Precise computations of chemotactic collapse using moving mesh
  methods.
\newblock {\em J. Comput. Phys.}, 202(2):463--487, 2005.

\bibitem{BCW10}
Martin Burger, Jos{\'e}~A. Carrillo, and Marie-Therese Wolfram.
\newblock A mixed finite element method for nonlinear diffusion equations.
\newblock {\em Kinet. Relat. Models}, 3(1):59--83, 2010.

\bibitem{BLT21}
Martin Burger, Philippe Lauren{\c{c}}ot, and Ariane Trescases.
\newblock Delayed blow-up for chemotaxis models with local sensing.
\newblock {\em J. Lond. Math. Soc., II. Ser.}, 103(4):1596--1617, 2021.

\bibitem{ChDr11}
Claire Chainais-Hillairet and J{\'e}r{\^o}me Droniou.
\newblock Finite-volume schemes for noncoercive elliptic problems with
  {Neumann} boundary conditions.
\newblock {\em IMA J. Numer. Anal.}, 31(1):61--85, 2011.

\bibitem{CEHK18}
Alina Chertock, Yekaterina Epshteyn, Hengrui Hu, and Alexander Kurganov.
\newblock High-order positivity-preserving hybrid
  finite-volume-finite-difference methods for chemotaxis systems.
\newblock {\em Adv. Comput. Math.}, 44(1):327--350, 2018.

\bibitem{CK08}
Alina Chertock and Alexander Kurganov.
\newblock A second-order positivity preserving central-upwind scheme for
  chemotaxis and haptotaxis models.
\newblock {\em Numer. Math.}, 111(2):169--205, 2008.

\bibitem{ChKi24}
Kyunghan Choi and Yong-Jung Kim.
\newblock Chemotactic cell aggregation viewed as instability and phase
  separation.
\newblock {\em Nonlinear Anal., Real World Appl.}, 80:15, 2024.
\newblock Id/No 104147.

\bibitem{ssPPV16}
Manuel del Pino, Angela Pistoia, and Giusi Vaira.
\newblock Large mass boundary condensation patterns in the stationary
  {Keller}-{Segel} system.
\newblock {\em J. Differ. Equations}, 261(6):3414--3462, 2016.

\bibitem{ssPiWe06}
Manuel del Pino and Juncheng Wei.
\newblock Collapsing steady states of the {Keller}-{Segel} system.
\newblock {\em Nonlinearity}, 19(3):661--684, 2006.

\bibitem{DKTY2019}
Laurent Desvillettes, Yong-Jung Kim, Ariane Trescases, and Changwook Yoon.
\newblock A logarithmic chemotaxis model featuring global existence and
  aggregation.
\newblock {\em Nonlinear Anal. Real World Appl.}, 50:562--582, 2019.

\bibitem{DLTW23}
Laurent Desvillettes, Philippe Lauren{\c{c}}ot, Ariane Trescases, and Michael
  Winkler.
\newblock Weak solutions to triangular cross diffusion systems modeling
  chemotaxis with local sensing.
\newblock {\em Nonlinear Anal., Theory Methods Appl., Ser. A, Theory Methods},
  226:26, 2023.
\newblock Id/No 113153.

\bibitem{DroNat18_super_conv}
J{\'e}r{\^o}me Droniou and Neela Nataraj.
\newblock Improved {{\(L^2\)}} estimate for gradient schemes and
  super-convergence of the {TPFA} finite volume scheme.
\newblock {\em IMA J. Numer. Anal.}, 38(3):1254--1293, 2018.

\bibitem{Ep12}
Yekaterina Epshteyn.
\newblock Upwind-difference potentials method for {Patlak}-{Keller}-{Segel}
  chemotaxis model.
\newblock {\em J. Sci. Comput.}, 53(3):689--713, 2012.

\bibitem{EI09}
Yekaterina Epshteyn and Ahmet Izmirlioglu.
\newblock Fully discrete analysis of a discontinuous finite element method for
  the keller-segel chemotaxis model.
\newblock {\em J. Sci. Comput.}, 40(1-3):211--256, 2009.

\bibitem{EK09}
Yekaterina Epshteyn and Alexander Kurganov.
\newblock New interior penalty discontinuous {Galerkin} methods for the
  {Keller}-{Segel} chemotaxis model.
\newblock {\em SIAM J. Numer. Anal.}, 47(1):386--408, 2009.

\bibitem{EpXi19}
Yekaterina Epshteyn and Qing Xia.
\newblock Efficient numerical algorithms based on difference potentials for
  chemotaxis systems in 3d.
\newblock {\em J. Sci. Comput.}, 80(1):26--59, 2019.

\bibitem{EGH00}
Robert Eymard, Thierry Gallou{\"e}t, and Rapha{\`e}le Herbin.
\newblock Finite volume methods.
\newblock In {\em Solution of equations in \(\mathbb{R}^n\) (Part 3).
  Techniques of scientific computing (Part 3)}, pages 713--1020. Amsterdam:
  North-Holland/ Elsevier, 2000.

\bibitem{Fil06}
Francis Filbet.
\newblock A finite volume scheme for the {Patlak}-{Keller}-{Segel} chemotaxis
  model.
\newblock {\em Numer. Math.}, 104(4):457--488, 2006.

\bibitem{FoLe07}
Irene Fonseca and Giovanni Leoni.
\newblock {\em Modern methods in the calculus of variations. {{\(L^p\)}}
  spaces}.
\newblock Springer Monogr. Math. New York, NY: Springer, 2007.

\bibitem{FuJi21b}
Kentaro Fujie and Jie Jiang.
\newblock Boundedness of classical solutions to a degenerate {Keller}-{Segel}
  type model with signal-dependent motilities.
\newblock {\em Acta Appl. Math.}, 176:36, 2021.
\newblock Id/No 3.

\bibitem{FuSe22}
Kentaro Fujie and Takasi Senba.
\newblock Global existence and infinite time blow-up of classical solutions to
  chemotaxis systems of local sensing in higher dimensions.
\newblock {\em Nonlinear Anal., Theory Methods Appl., Ser. A, Theory Methods},
  222:7, 2022.
\newblock Id/No 112987.

\bibitem{FuJi20}
Kentarou Fujie and Jie Jiang.
\newblock Global existence for a kinetic model of pattern formation with
  density-suppressed motilities.
\newblock {\em J. Differ. Equations}, 269(6):5338--5378, 2020.

\bibitem{FuJi21a}
Kentarou Fujie and Jie Jiang.
\newblock Comparison methods for a {Keller}-{Segel}-type model of pattern
  formations with density-suppressed motilities.
\newblock {\em Calc. Var. Partial Differ. Equ.}, 60(3):37, 2021.
\newblock Id/No 92.

\bibitem{Grisvard}
Pierre Grisvard.
\newblock {\em Elliptic problems in nonsmooth domains}, volume~69 of {\em
  Class. Appl. Math.}
\newblock Philadelphia, PA: Society for Industrial {and} Applied Mathematics
  (SIAM), reprint of the 1985 hardback ed. edition, 2011.

\bibitem{Grog93}
Konrad Gr{\"o}ger.
\newblock Boundedness and continuity of solutions to linear elliptic boundary
  value problems in two dimensions.
\newblock {\em Math. Ann.}, 298(4):719--727, 1994.

\bibitem{GLY19}
Li~Guo, Xingjie~Helen Li, and Yang Yang.
\newblock Energy dissipative local discontinuous {Galerkin} methods for
  {Keller}-{Segel} chemotaxis model.
\newblock {\em J. Sci. Comput.}, 78(3):1387--1404, 2019.

\bibitem{HeNeVa23}
Federico Herrero-Herv{\'a}s, Mihaela Negreanu, and Antonio~Manuel Vargas.
\newblock Convergence of a meshless numerical method for a chemotaxis system
  with density-suppressed motility.
\newblock {\em Comput. Math. Appl.}, 148:293--301, 2023.

\bibitem{JiLa21}
Jie Jiang and Philippe Lauren{\c{c}}ot.
\newblock Global existence and uniform boundedness in a chemotaxis model with
  signal-dependent motility.
\newblock {\em J. Differ. Equations}, 299:513--541, 2021.

\bibitem{JiLaZh22}
Jie Jiang, Philippe Lauren{\c{c}}ot, and Yanyan Zhang.
\newblock Global existence, uniform boundedness, and stabilization in a
  chemotaxis system with density-suppressed motility and nutrient consumption.
\newblock {\em Commun. Partial Differ. Equations}, 47(5):1024--1069, 2022.

\bibitem{JiWa20}
Hai-Yang Jin and Zhi-An Wang.
\newblock Critical mass on the {Keller}-{Segel} system with signal-dependent
  motility.
\newblock {\em Proc. Am. Math. Soc.}, 148(11):4855--4873, 2020.

\bibitem{KeSe70}
Evelyn~F. Keller and Lee~A. Segel.
\newblock Initiation of slime mold aggregation viewed as an instability.
\newblock {\em J. Theoret. Biol.}, 26(3):399--415, 1970.

\bibitem{KeSe71a}
Evelyn~F. Keller and Lee~A. Segel.
\newblock Model for chemotaxis.
\newblock {\em J. Theoret. Biol.}, 30:225--234, 1971.

\bibitem{KuOg20}
Masaki Kurokiba and Takayoshi Ogawa.
\newblock Singular limit problem for the {K}eller-{S}egel system and
  drift-diffusion system in scaling critical spaces.
\newblock {\em J. Evol. Equ.}, 20(2):421--457, 2020.

\bibitem{KuOg20b}
Masaki Kurokiba and Takayoshi Ogawa.
\newblock Singular limit problem for the two-dimensional {K}eller-{S}egel
  system in scaling critical space.
\newblock {\em J. Differential Equations}, 269(10):8959--8997, 2020.

\bibitem{KuOg22}
Masaki Kurokiba and Takayoshi Ogawa.
\newblock Maximal regularity and a singular limit problem for the
  {P}atlak-{K}eller-{S}egel system in the scaling critical space involving
  {$BMO$}.
\newblock {\em Partial Differ. Equ. Appl.}, 3(1):Paper No. 3, 56, 2022.

\bibitem{Lem13}
Pierre~Gilles Lemari\'e-Rieusset.
\newblock Small data in an optimal {B}anach space for the parabolic-parabolic
  and parabolic-elliptic {K}eller-{S}egel equations in the whole space.
\newblock {\em Adv. Differential Equations}, 18(11-12):1189--1208, 2013.

\bibitem{LXZ23}
Min Li, Zhaoyin Xiang, and Guanyu Zhou.
\newblock The stability analysis of a 2{D} {K}eller-{S}egel-{N}avier-{S}tokes
  system in fast signal diffusion.
\newblock {\em European J. Appl. Math.}, 34(1):160--209, 2023.

\bibitem{LSY17}
Xingjie~Helen Li, Chi-Wang Shu, and Yang Yang.
\newblock Local discontinuous {Galerkin} method for the {Keller}-{Segel}
  chemotaxis model.
\newblock {\em J. Sci. Comput.}, 73(2-3):943--967, 2017.

\bibitem{Lip14}
Friedrich Lippoth.
\newblock On the justification of the quasistationary approximation of several
  parabolic moving boundary problems---{P}art {I}.
\newblock {\em Nonlinear Anal. Real World Appl.}, 17:1--22, 2014.

\bibitem{LWZ18}
Jian-Guo Liu, Li~Wang, and Zhennan Zhou.
\newblock Positivity-preserving and asymptotic preserving method for 2{D}
  {K}eller-{S}egal equations.
\newblock {\em Math. Comp.}, 87(311):1165--1189, 2018.

\bibitem{Ma03}
Americo Marrocco.
\newblock Numerical simulation of chemotactic bacteria aggregation via mixed
  finite elements.
\newblock {\em M2AN, Math. Model. Numer. Anal.}, 37(4):617--630, 2003.

\bibitem{Miz19}
Masaaki Mizukami.
\newblock The fast signal diffusion limit in a {K}eller-{S}egel system.
\newblock {\em J. Math. Anal. Appl.}, 472(2):1313--1330, 2019.

\bibitem{NoSa24}
Toru Nogayama and Yoshihiro Sawano.
\newblock Singular limit problem for the {K}eller-{S}egel system and
  drift-diffusion system in scaling critical {B}esov-{M}orrey spaces.
\newblock {\em J. Math. Anal. Appl.}, 529(2):Paper No. 127207, 51, 2024.

\bibitem{OgSu23}
Takayoshi Ogawa and Takeshi Suguro.
\newblock Maximal regularity of the heat evolution equation on spatial local
  spaces and application to a singular limit problem of the {K}eller-{S}egel
  system.
\newblock {\em Math. Ann.}, 387(1-2):389--431, 2023.

\bibitem{OtSt97}
Hans~G. Othmer and Angela Stevens.
\newblock Aggregation, blowup, and collapse: the {ABC}s of taxis in reinforced
  random walks.
\newblock {\em SIAM J. Appl. Math.}, 57(4):1044--1081, 1997.

\bibitem{ssPiVa15}
Angela Pistoia and Giusi Vaira.
\newblock Steady states with unbounded mass of the {Keller}-{Segel} system.
\newblock {\em Proc. R. Soc. Edinb., Sect. A, Math.}, 145(1):203--222, 2015.

\bibitem{Rac09}
Andrzej Raczy\'nski.
\newblock Stability property of the two-dimensional {K}eller-{S}egel model.
\newblock {\em Asymptot. Anal.}, 61(1):35--59, 2009.

\bibitem{Sa07}
Norikazu Saito.
\newblock Conservative upwind finite-element method for a simplified
  {Keller}-{Segel} system modelling chemotaxis.
\newblock {\em IMA J. Numer. Anal.}, 27(2):332--365, 2007.

\bibitem{Sa09}
Norikazu Saito.
\newblock Conservative numerical schemes for the {Keller}-{Segel} system and
  numerical results.
\newblock {\em RIMS K{\^o}ky{\^u}roku Bessatsu}, B15:125--146, 2009.

\bibitem{Sa12}
Norikazu Saito.
\newblock Error analysis of a conservative finite-element approximation for the
  {Keller}-{Segel} system of chemotaxis.
\newblock {\em Commun. Pure Appl. Anal.}, 11(1):339--364, 2012.

\bibitem{SaSu05}
Norikazu Saito and Takashi Suzuki.
\newblock Notes on finite difference schemes to a parabolic-elliptic system
  modelling chemotaxis.
\newblock {\em Appl. Math. Comput.}, 171(1):72--90, 2005.

\bibitem{ssSchaaf85}
Renate Schaaf.
\newblock Stationary solutions of chemotaxis systems.
\newblock {\em Trans. Am. Math. Soc.}, 292:531--556, 1985.

\bibitem{ssSeSu00}
Takasi Senba and Takashi Suzuki.
\newblock Some structures of the solution set for a stationary system of
  chemotaxis.
\newblock {\em Adv. Math. Sci. Appl.}, 10(1):191--224, 2000.

\bibitem{SSKHT13}
Robert Strehl, Andriy Sokolov, Dmitri Kuzmin, Dirk Horstmann, and Stefan Turek.
\newblock A positivity-preserving finite element method for chemotaxis problems
  in 3d.
\newblock {\em J. Comput. Appl. Math.}, 239:290--303, 2013.

\bibitem{ssWaWe02}
Guofang Wang and Jun-Cheng Wei.
\newblock Steady state solutions of a reaction-diffusion system modeling
  chemotaxis.
\newblock {\em Math. Nachr.}, 233-234:221--236, 2002.

\bibitem{WWX19}
Yulan Wang, Michael Winkler, and Zhaoyin Xiang.
\newblock The fast signal diffusion limit in {K}eller-{S}egel(-fluid) systems.
\newblock {\em Calc. Var. Partial Differential Equations}, 58(6):Paper No. 196,
  40, 2019.

\bibitem{ssWaXu21}
Zhi-An Wang and Xin Xu.
\newblock Steady states and pattern formation of the density-suppressed
  motility model.
\newblock {\em IMA J. Appl. Math.}, 86(3):577--603, 2021.

\bibitem{ZZZ16}
Jiansong Zhang, Jiang Zhu, and Rongpei Zhang.
\newblock Characteristic splitting mixed finite element analysis of
  {Keller}-{Segel} chemotaxis models.
\newblock {\em Appl. Math. Comput.}, 278:33--44, 2016.

\bibitem{ZZLY16}
Rongpei Zhang, Jiang Zhu, Abimael F.~D. Loula, and Xijun Yu.
\newblock Operator splitting combined with positivity-preserving discontinuous
  {Galerkin} method for the chemotaxis model.
\newblock {\em J. Comput. Appl. Math.}, 302:312--326, 2016.

\bibitem{ZS17}
Guanyu Zhou and Norikazu Saito.
\newblock Finite volume methods for a {Keller}-{Segel} system: discrete energy,
  error estimates and numerical blow-up analysis.
\newblock {\em Numer. Math.}, 135(1):265--311, 2017.

\end{thebibliography}

\end{document}